%% file: high_order_nls_hal.tex
\documentclass[twoside]{article}
\usepackage{amssymb,amsthm,amsmath}
\usepackage{xcolor,graphicx}
\usepackage{subfigure}
\usepackage{float}
\usepackage{url}
\usepackage{verbatim}
\usepackage{tikz}
\usepackage{pgfplots}
\pgfplotsset{compat=newest}
\pgfplotsset{plot coordinates/math parser=false}
\newlength\figureheight
\newlength\figurewidth 

\usepackage[top=2cm,bottom=2cm,left=2cm,right=2cm,headsep=10pt,a4paper]{geometry} 

\usepackage[T1]{fontenc}
\usepackage{lmodern}

\makeatletter

\@addtoreset{equation}{section}
\makeatother

\newtheorem{theorem}{Theorem} 
\newtheorem{lemma}[theorem]{Lemma} 
 
\newtheorem{remark}[theorem]{Remark} 
\newtheorem{corollary}[theorem]{Corollary}
\newtheorem{definition}[theorem]{Definition}

\newcommand{\C}{{\mathbb C}}
\newcommand{\R}{{\mathbb R}} 
\newcommand{\Z}{{\mathbb Z}} 
\newcommand{\N}{{\mathbb N}} 
 
\newcommand{\T}{{\mathbb T}} 

\newcommand{\ds}{\displaystyle}

\usepackage{fancyhdr}
 
\title{High order exponential integrators for nonlinear
Schr\"odinger equations with application to rotating
Bose-Einstein condensates}
\author{{\bf C. Besse} \\
Institut de Math\'ematiques de  Toulouse UMR5219,\\
Universit\'e de Toulouse; CNRS, UPS IMT;\\
F-31062 Toulouse Cedex 9, France.\\
	\texttt{Christophe.Besse@math.univ-toulouse.fr} \\\\\
{\bf G. Dujardin} \\
	Inria Lille Nord-Europe, MEPHYSTO Team\\ 
  	59650 Villeneuve d'Ascq, France, and\\
        Laboratoire Paul Painlev\'e, UMR CNRS 8525\\
        \texttt{guillaume.dujardin@inria.fr} \\\\
{\bf I. Lacroix-Violet} \\
	Laboratoire Paul Painlev\'e, Universit\'e Lille Nord de France,\\ 
  	CNRS UMR 8524, INRIA RAPSODI Team,\\
  	Universit\'e Lille 1 Sciences et Technologies,
  	Cit\'e Scientifique,\\ 59655 Villeneuve d'Ascq Cedex, France.\\
	\texttt{Ingrid.Violet@math.univ-lille1.fr}}

\begin{document}
\maketitle

\begin{abstract}
This article deals with the numerical integration in time
of nonlinear Schr\"odinger equations.
The main application is the numerical simulation of rotating Bose-Einstein
condensates.
The authors perform a change of unknown so that the rotation term
disappears and they obtain as a result a nonautonomous nonlinear Schr\"odinger
equation.
They consider exponential integrators such as exponential Runge--Kutta methods
and Lawson methods. They provide an analysis of the order of convergence
and some preservation properties
of these methods in a simplified setting and they supplement their
results with numerical experiments with realistic physical parameters.
Moreover, they compare these methods with the classical split-step methods
applied to the same problem.
\end{abstract}

{\small\noindent 

} 

\section{Introduction}

This paper deals with the numerical integration of nonautonoumous nonlinear Schr\"odinger (NLS) equations which read
\begin{equation}
  \label{eq:schro}
  \left \{
    \begin{array}{ll}
      i \partial_t \psi=-\frac{1}{2} \Delta \psi + V(t,\mathbf{x}) \psi + \beta |\psi|^{2\kappa}  \psi, & (t,\mathbf{x}) \in \mathbb{R}^+\times \mathbb{R}^d,\\
      \psi(0,\mathbf{x})=\psi_0(\mathbf{x}), & x \in \R^d.
    \end{array}
  \right .
\end{equation}
The operator $\Delta$ denotes the usual Laplace operator on $\R^d$, $d\in \mathbb{N}^*$. The potential function $V$ is smooth and $\kappa\in \N$.
The unknown $\psi$ is a complex-valued wavefunction associated to the given initial datum $\psi_0$. The derivation and the analysis of efficient semi-discrete numerical methods for the time integration of these equations have a long history. Some authors are interested in the finite time accuracy of the schemes \cite{DuSa2000, DeFoPe2981, BeBiDe2002, Lu2008}. With additional hypotheses corresponding to physically relevant situations, equation \eqref{eq:schro} may have several invariants such as energy, momentum, etc, ..., in addition to mass conservation. Several authors show interest in the preservation of the invariants by the numerical schemes \cite{DeFoPe2981, Be2004, WeHe1986, FePeVa1995, SaVe1986}. Beyond finite time integration of Eq. \eqref{eq:schro}, asymptotic regimes have been considered: long time preservation of invariants \cite{CoGa2013, GaLu2012, DuFa2007} and semi-classical regimes \cite{BaJiMa2002, CaMo2011, BeCaMe2013, CaChMeMu2014, ChCrLeMe2014}.

Our goal in this paper is to introduce and analyse two classes of exponential methods for the time integration of \eqref{eq:schro} on a $d$-dimensional space. The first class is that of exponential Runge--Kutta methods \cite{HoOs2005_2, HoOs2005, Du2009}. The second class relies on the Lawson techniques \cite{La1967, Ro2001}. 
We focus on situations were the dynamics of equation \eqref{eq:schro} essentially stays in a bounded domain of $\R^d$. Hence, we replace Eq. \eqref{eq:schro} with the same equation on a large periodic torus {\color{black}$\T_\delta^{d}$ (with characteristic size $\delta>0$). Such a change is usually not uniform with respect to $\delta$ (see subsection \ref{subsec:simplifiedsetting}). However it is physically relevant even if it is only valid for reasonably bounded times
{\it a priori}.} Our main application is the Gross-Pitaevskii equation modeling a rotating Bose-Einstein Condensate (BEC) in $\R^2$ for which the dynamics of the solutions is much spatialy localized. We prove our results for functions in Sobolev spaces $H^\sigma(\T_\delta^d)$ for $\sigma>d/2$, {\color{black} instead of
$H^\sigma(\R^d)$}. The fully discrete corresponding situations are not similar.
Discretizing functions on the whole space requires to deal with carefully
chosen specific boundary conditions.
Working on a torus has the numerical advantage of avoiding boundary conditions.
Moreover the numerical computation of the linear group generated
by $i\Delta$ is much simpler in the case of the torus \cite{HoLu1996}.
{\color{black}An alternative to the reduction from $\R^d$ to $\T_\delta^d$ consists in using pseudospectral methods introduced in \cite{BaLiSh2009} and analyzed in \cite{HoKoTh2014}. In contrast to the approach presented in this paper, pseudospectral methods can not be used for general confining potentials ({\it e.g.} quadratic potentials without symmetry in the rotation plane ($\gamma_x^2\neq\gamma_y^2$ in \eqref{potharmonic}) that result in non-autonomous linear parts after the change of variable introduced below (see \eqref{GPeqnv2}) or quadratic-quartic potentials such as
$V(x)=\|x\|^2+(1+\sin(\|x\|^2))\|x\|^4$).}

Before the introduction of the methods (Sections \ref{expRK} and \ref{Lawsonmethods}) and our numerical results (Section \ref{Comparaisons}), we introduce below our main application and show how it fits equation \eqref{eq:schro}.
 
\subsection{Presentation of the application}

A Bose-Einstein Condensate is the state of matter reached by a dilute gas of bosons cooled to very low temperature. A large fraction of bosons occupies the lowest quantum state so that macroscopic quantum phenomena become apparent. This phenomena was theoretically predicted by Bose in 1924 for photons \cite{Bo1924} and generalized to atoms by Einstein in 1925 \cite{Ei1925}. The first experimental evidence of BEC was obtained in 1995 \cite{AnEnMaWiCo1995, DaMeAnDrDuKuKe1995}.

At low temperature, a rotating planar BEC can be described by the macroscopic complex-valued wave function $\varphi=\varphi(t,\mathbf{x})$ whose evolution is governed by a Gross-Pitaevskii Equation (GPE) with an angular momentum rotation term. After a suitable changes of variables \cite{BMTZ2013}, the dimensionless GPE in $d=2$ dimensions satisfied by $\varphi$ can be written for $\mathbf{x} \in \R^{2}$:
\begin{equation}
\label{GPEeq}
i\partial_t \varphi=-\frac{1}{2}\Delta \varphi + V_{c}(\mathbf{x}) \varphi+ \beta |\varphi|^2\varphi-\Omega R\varphi.
\end{equation} 
The real-valued function $V_{c}=V_{c}(\mathbf{x})$ corresponds to a
{\color{black} smooth} 
potential depending only on the space variables denoted by
$\mathbf{x}=(x,y)^{t}$.
In our physical context, this potential is confining: this means
that $V_{c}(\mathbf{x})$ tends to $+ \infty$ when $\|\mathbf x\|=\sqrt{x^2+y^2}$
tends to $+\infty$. For example, in this paper, we consider potentials
of the form
\begin{equation}
\label{potharmonic}
V_{c}(\mathbf{x})=\frac{1}{2}\left(\gamma_{x}^{2}x^{2}+\gamma_{y}^{2}y^{2}\right),
\end{equation}
where $\gamma_x,\gamma_y>0$.
This confining potential competes with the rotation operator
$-\Omega R=i\Omega(x\partial_y-y\partial_x)$ at angular speed $\Omega\in\R$:
the former tends to 
make bosons stay together at the origin of the plane, while the latter
tends to spread the bosons out.
The real coefficient $\beta$ represents the nonlinearity strength,
and comes from the averaged effect of the bosons.
The evolution equation \eqref{GPEeq} is supplemented with an initial condition
\begin{equation}
\label{GPEeqCI}
\varphi(0,\mathbf{x})=\varphi_{0}(\mathbf{x}), \quad \hbox{ for all }~\mathbf{x} \in \R^{2}.
\end{equation}
If one introduces the functional
\begin{equation}
\label{energie}
E(\varphi)=\int_{\R^{2}} \left[\frac{1}{2}|\nabla \varphi|^{2}+V_{c}|\varphi|^{2}+\frac{\beta}{2}|\varphi|^{4}-\Omega Re(\varphi^{\star}R\varphi)\right]d\mathbf{x},
\end{equation}
then \eqref{GPEeq} reads
\begin{equation*}
  i\partial_t\varphi(t,.) = \nabla_{\varphi(t,.)} E(\varphi(t,.)).
\end{equation*}
The equation happens to be Hamiltonian, and the energy is preserved
by the dynamics: along a solution $t\mapsto \varphi(t,.)$
of \eqref{GPEeq}-\eqref{GPEeqCI}, one has $E(\varphi(t,.))=E(\varphi_0)$.
In addition to preserving the energy, the evolution preserves the
mass of the wave function: along a solution $t\mapsto \varphi(t,.)$
of \eqref{GPEeq}-\eqref{GPEeqCI}, one has $\int_{\R^{2}} |\varphi(t,\mathbf{x})|^2d\mathbf{x} = \int_{\R^{2}} |\varphi_0(\mathbf{x})|^2d\mathbf{x}$.
Another dynamical feature of this equation is the evolution of the
angular momentum expectation: if one denotes by
\begin{equation}
\label{angular}
<R>(t)=\int_{\R^{2}} \varphi^{\star}(t,\mathbf{x})R\varphi(t,\mathbf{x})d\mathbf{x},
\end{equation}
then this real-valued function is constant in the special case of
a radial harmonic potential ($\gamma_x=\gamma_y$ in \eqref{potharmonic})
and has a more complex dynamics in more general
cases (see Lemma 6.2.1 in \cite{TPHD2013}).

In the last decades this model has been studied a lot \cite{BZ2005, BDZ2006,AnDu2015,AnBaBe2013, TPHD2013}. 
An important issue in the numerical time integration of equations
\eqref{GPEeq}-\eqref{GPEeqCI} comes from the rotation term $R$. Following \cite{BMTZ2013}, we introduce new coordinates that allow {\color{black}us} to put equation \eqref{GPEeq} in the form of equation \eqref{eq:schro}. 
Let us set for $t\in\R$,
\begin{equation}
\label{matnewvariables}
A(t)=\begin{pmatrix}
\cos(\Omega t) & -\sin(\Omega t) \\
\sin(\Omega t) & \cos(\Omega t)
\end{pmatrix}.
\end{equation}
Note that $A(t)$ is orthogonal and hence satisfies $A(t)^{-1}=A(t)^t$.
We perform the change of unknown $\varphi\leftrightarrow\psi$ defined by
\begin{equation}
\label{newvariables}
\varphi(t,\mathbf{x})=\psi \left(t,A(t)\mathbf{x}\right).
\end{equation}
This way, $\varphi$ solves \eqref{GPEeq} if and only if $\psi$
solves
\begin{equation}
\label{GPeqnv2}
\partial_t \psi=\frac{i}{2}\Delta \psi -i V(t,\tilde{\mathbf{x}}) \psi-i \beta |\psi|^2\psi,
\end{equation}  
where $V$ is a time dependent potential given by
\begin{equation*}
V(t,\tilde{\mathbf{x}})=V_{c}\left(A(t)\tilde{\mathbf{x}}\right),
\end{equation*}
and the initial datum is
\begin{equation*}
  \psi(0,\tilde{\bf x})=\psi_0(\tilde {\bf x})=\varphi_0(\tilde{\bf x}).
\end{equation*}
For convenience, we shall denote explicitly
the following time-dependent change of spatial
variables:
\begin{equation*}
\tilde{\mathbf{x}}=
\begin{pmatrix}
\tilde x \\ \tilde y
\end{pmatrix}= A(t) 
\begin{pmatrix}
x \\ y
\end{pmatrix}=A(t)\mathbf{x}.
\end{equation*}

Note that equation \eqref{GPeqnv2} satisfied by $\psi$ is a standard nonlinear Schr\"odinger equation \eqref{eq:schro} with a cubic nonlinearity ($\kappa=1$) and a space and time dependent potential.

\subsection{General setting of the Cauchy problem}
\label{subsec:simplifiedsetting}

The unboundedness of the potential function $V$ makes the numerical analysis of the exponential methods difficult. 
Therefore, we modify the equation \eqref{eq:schro} by cutting off the potential $V$ smoothly.
Let us motivate this modification.
In the context of Schr\"odinger equations with confining potentials,
if the mass of the initial datum is essentially concentrated in a bounded set around the
origin, then
this mass localisation property
is to be preserved by the evolution of \eqref{eq:schro}, at least
for reasonable times. Therefore, modifying the potential $V$
out of a sufficiently large bounded set around the origin will not create
huge errors in the solution, at least for not too long times.
Let us introduce a smooth function $\chi : \R\rightarrow [0,1]$ such that
\begin{equation*}
  \forall x\in [1-\delta/2, \delta/2-1],\quad \chi(x)=1
  \quad \text{ and } \quad
  \forall x\in (-\infty,-\delta/2)\cup(\delta/2,+\infty), \quad \chi(x)=0,
\end{equation*}
where $\delta\gg 2$ is a given real number, chosen accordingly
with the initial datum $\varphi_0$, the other
physical parameters and the
final computational time $T>0$.
We define the new potential function $W$ for $t\in[0,T]$,
$\mathbf{x}=(x_1,\cdots,x_d)^t$ by letting
\begin{equation}
\label{defW}
  W(t,\mathbf{x})
= V(t,\mathbf{x}) \prod_{j=1}^d \chi(x_j).
\end{equation}
Although modifying the potential function as above changes deeply the
physical situation, the fact that $\delta$ is taken accordingly
to $\varphi_0$ and the physical parameters and for a finite time interval
$[0,T]$ gives some hope in the fact that the evolution
of equation \eqref{eq:schro} with the potential function $V$ and
that of the same equation with $V$ replaced with $W$ starting
from the same initial datum $\varphi_0$ will be quite similar.

To be more specific with the periodization of the problem, let us denote by $\T_\delta$ the quotient $\R/(\delta \Z)$.
We consider the function $w$ defined as $\delta$-periodic in all
directions such that for $t\in [0,T]$ and all $\mathbf{x} \in \R^d$
satisfying $|x_j|\leq \delta/2$, $1\leq j \leq d$,
one has $w(t,\mathbf{x})=W(t,\mathbf{x})$.
The mapping $t\mapsto w(t,.)$ is smooth from $[0,T]$
to $H^\sigma(\T_\delta^d)$ as soon as $\sigma \geq 0$.
In the following sections, we replace the continuous
problem \eqref{eq:schro} with its periodic counterpart as explained above,
and we assume that $\varphi_{0}=\psi_0\in H^{\sigma}(\T_\delta^d)$ for some $\sigma>d/2$.

We therefore get
the following semilinear
Cauchy problem in time:
\begin{equation}
\label{Cauchy1}
\begin{array}{ll}
\partial_{t} \psi(t,{\mathbf{x}}) -L \psi(t,{\mathbf{x}})
=N_{w}(t,\psi)({\mathbf x}),
\quad (t,{\mathbf{x}}) \in [0,T]\times\T_\delta^d, \\[2mm]
\psi(0,{\mathbf{x}})=\psi_{0}({\mathbf{x}}),
\quad {\mathbf{x}} \in \T_\delta^d,
\end{array}
\end{equation}
where $T>0, ~L=\ds\frac{i\Delta}{2}$ and
\begin{equation}
\label{eq:defNper}
N_{w}(t,\psi)({\mathbf x})=
-iw(t,{\mathbf{x}})\psi(t,{\mathbf x})
-i\beta |\psi|^{2\kappa}\psi(t,{\mathbf x}).
\end{equation}

\begin{remark}\label{rmkchoicecauchypb}
The choice of the definition of the linear part $L$ and the nonlinear part
$N$ from equation \eqref{eq:schro} is somewhat arbitrary.
One could also choose, for example
\begin{equation*}
  L=\ds\frac{i\Delta}{2} -iw(t,{\mathbf{x}})
  \qquad {\rm and}
  \qquad
   N(t,\psi)({\mathbf x})=-i \beta |\psi|^{2\kappa}\psi(t,{\mathbf x}),
\end{equation*}
but this would lead to a nonautonomous linear problem whose spectral properties
are not as nice as that of the other case.
{\color{black} However,
in the case of a radially symmetric potential, one may get an autonomous
linear part and use appropriate spectral methods on the whole space $\R^d$
(see Eq. (2d) in \cite{HoKoTh2014}).}
{\color{black}
In this paper, we consider fairly general potentials without symmetry
(see subsection \ref{Comp2d} for numerical experiments with $\gamma_x\neq\gamma_y$
in \eqref{potharmonic})
and}
we will indeed use the nice spectral properties of the operator $L=i\Delta/2$
later.
\end{remark}

\noindent
One can check that the operators of the form $e^{\alpha L}$ for $\alpha\in\R$
are isometric over $H^\sigma(\T_\delta^d)$.
The function $t\mapsto w(t,\cdot)$ belongs to
${\mathcal C}^{\infty}(\R,H^\sigma(\T_\delta^d))$.
Moreover, the nonlinear function $N_w$
satisfies a local Lipschitz condition:
\begin{lemma}
\label{lemmaLipschitz}
  For all $T>0$, for all $r>0$, there exists a constant $C>0$ such that
for all $t\in [0,T]$, $\varphi_1,\varphi_2\in H^\sigma(\T_\delta^d)$ such that
$\|\varphi_1\|_{H^\sigma(\T_\delta^d)}\leq r$ and
$\|\varphi_2\|_{H^\sigma(\T_\delta^d)}\leq r$,
we have $\|N_w(t,\varphi_1)-N_w(t,\varphi_2)\|_{H^\sigma(\T_\delta^d)} \leq
    C \|\varphi_1-\varphi_2\|_{H^\sigma(\T_\delta^d)}$.
\end{lemma}
\begin{remark}
  Note that this Lipschitz property is also true if one replaces $N_w$
with $N_W(t,\psi)=-iW(t,\cdot)\psi(\cdot)-i\beta|\psi|^{2\kappa}\psi(\cdot)$
and the norms are taken in $H^\sigma(\R^d)$ with $\sigma>d/2$,
but is no longer true with
$N_{ V}=-i V(t,\cdot)\psi(\cdot)-i\beta|\psi|^{2\kappa}\psi(\cdot)$.
\end{remark}
We recover the fact that the Cauchy problem \eqref{Cauchy1}
is well-posed in $H^\sigma(\T_\delta^d)$.

\subsection{Outline of the paper}

The outline of the paper is as follows: Section 2 is devoted to exponential
Runge--Kutta methods. Firstly, we briefly describe the construction of these
exponential methods, and recall what is the socalled underlying Runge--Kutta
method {\color{black}\cite{HoOs2010}}, which is defined using $s$ collocation points.
Secondly, we prove for the equation \eqref{Cauchy1} that if the
$s$ collocation points are distinct, then the exponential Runge-Kutta method applied to Problem
\eqref{Cauchy1} has order $s$.
{\color{black}Moreover, we observe numerically
that if we use Gauss collocation points,
then we obtain order $2s$ (superconvergence).}
Section 3 is devoted to exponential integrators named Lawson methods.
These methods are collocation methods on a new evolution equation, obtained
from Problem \eqref{Cauchy1} {\it via} another change of unknown.
We show that Lawson methods applied to the simplified equation
keep their classical order.
In particular, the methods with $s$ stages defined with Gauss points
have order $2s$ {\color{black}(superconvergence)}.
Moreover, they preserve quadratic invariants up to round off errors.
Section 4 is devoted to the comparison of our methods with other methods such
as splitting methods as used in \cite{BMTZ2013}.
For our numerical experiments in dimension 1 and 2,
the algorithm in time is supplemented with
the fast Fourier transform (FFT) method in space, in a periodic domain.
Finally we end the paper with conclusion and outlook. 

\section{Exponential Runge--Kutta methods}
\label{expRK}

The main idea behind exponential integrators is to integrate exactly
the linear part of the problem and then to use an appropriate approximation
of the nonlinear part.
Exponential Runge--Kutta (ERK) methods are particular exponential integrators.
These methods have been derived and analysed for semi-linear parabolic
Cauchy problems (see for example \cite{HoOs2005} for collocation methods,
\cite{HoOs2005_2} for explicit methods and \cite{HoOs2010}
for a survey on exponential integrators).
They also have been used in \cite{Du2009} for solving linear and semi-linear
Schr\"odinger Cauchy problems on the $d$-dimensional torus. {\color{black} They have been used to solve nonlinear Schr\"odinger equations for optical fibers (see the interaction picture method analyzed in \cite{BaFeMaMeTe2016}).}

In this section, we introduce, analyse and use ERK methods
to solve numerically equation \eqref{GPeqnv2}.

\subsection{Notations and description of the method}

In order to solve numerically problem \eqref{Cauchy1},
we consider the following ERK methods of collocation type.
We refer to \cite{HoOs2005} for a derivation of such methods for semi-linear problems based on variation-of-constants formula.

Let $T>0$ be the final computational time and $(t_{n})_{0\leq n \leq M}$
be a uniform subdivision of $[0,T]$ with $M+1$ points
{\it i.e.} $t_{n}=nh$ with $h=T/M$. Let $s \in \N^{\star}$ and $c_{1}, \cdots, c_{s} \in [0,1]$ be given such that for all
$(i,j) \in \{1,\cdots,s\}^{2}, c_{i}\neq c_{j}$ if $i\neq j$.
For some $n \in \{0,\cdots,M-1\}$, we assume we have an approximation
$\psi_{n}$ of the exact solution $\psi(t_{n})$ of problem \eqref{Cauchy1}
at time $t_n$.
Using an ERK method consists in computing an approximation
$\psi_{n+1}$ from  $\psi_{n}$ in the following way.
First, we solve the nonlinear system of $s$ equations
\begin{equation}
\label{nonlinsys}
\psi_{n,k}=e^{c_{k}hL}\psi_{n} + h
\sum_{\ell=1}^{s} a_{k,\ell}(hL)N_w(t_{n}+c_{\ell}h,\psi_{n,\ell}),
\quad 1 \leq k \leq s,
\end{equation}
where the unknowns are $\psi_{n,1}, \cdots, \psi_{n,s}$ and the $s^{2}$
coefficients $\left(a_{k,l}(hL)\right)_{(k,l)\in \{1,\cdots,s\}^{2}}$ are
linear continuous operators on $H^\sigma(\T_\delta^d)$ defined by
\begin{equation}
\label{coeffRK1}
a_{k,\ell}(hL)=\frac{1}{h} \int_{0}^{c_{k}h}
e^{(c_{k}h-\sigma)L}\mathcal{L}_{\ell}(\sigma) d\sigma,
\end{equation}
with $(\mathcal{L}_{\ell})_{1\leq \ell \leq s}$ the Lagrange polynomials defined by
\begin{equation}
\label{Lagrange}
\mathcal{L}_{\ell}(\tau)
=\prod_{j=1,j\neq \ell}^{s} \frac{\tau/h-c_{j}}{c_{\ell}-c_{j}},\quad 1 \leq \ell \leq s.
\end{equation}
Then we compute $\psi_{n+1}$ using
\begin{equation}
\label{Nexttime}
\psi_{n+1}=e^{hL}\psi_{n}
+ h \sum_{k=1}^{s} b_{k}(hL)N_w(t_{n}+c_kh,\psi_{n,k}),
\end{equation}
where
\begin{equation}
\label{coeffRK2}
\forall k\in\{1,\dots,s\},\quad
b_{k}(hL)=\frac{1}{h} \int_{0}^{h} e^{(h-\sigma)L}\mathcal{L}_{k}(\sigma)d\sigma.
\end{equation}

\begin{remark}
\label{remarkunderlyingRK}
  If we set $L=0$ in formulae \eqref{coeffRK1} and \eqref{coeffRK2},
then one recovers classical formulae defining Runge-Kutta collocation methods.
The Runge-Kutta collocation method defined by the corresponding
coefficients is called the underlying Runge-Kutta method
{\color{black}\cite{HoOs2010}}.
\end{remark}

Note that in order to be able to compute $\psi_{n+1}$ from $\psi_n$,
we have to precompute the coefficients $a_{k,\ell}(hL)$ and $b_{k}(hL)$
for all $k,\ell \in \{1,\cdots, s\}$.
We present an accurate and efficient way of precomputing these coefficents
in the next subsection.
Of course, the spatial discretization of the operator $L$ has to be
specified.

\subsection{Precomputation of the coefficients}
\label{Precompcoeffs}

It is well known that the {\color{black}Laplacian} operator 
on $H^\sigma(\T_\delta^d)$ is self-adjoint with eigenvalues
\begin{equation*}
\omega_{p}=-\left(\frac{2\pi}{\delta}\right)^{2}(p_1^{2}+\dots+p_d^{2}), \quad p=(p_1,\dots,p_d) \in \Z^{d}.
\end{equation*}
We discretize the problem in space using a uniform grid with $K^2$
points and rely on FFT techniques.
The eigenvalues of the discretized version of the periodic 
{\color{black}$L=i\Delta/2$} operator
are $i\omega_{p}/2$ for $0\leq |p|_\infty\leq K-1$, where $|p|_\infty={\rm max}\{|p_1|,\dots,|p_d|\}$.
Since this operator is also self-adjoint, the computation of the discretized
versions of the operators $a_{k,\ell}$ and $b_k$ amounts to the computation
of the values of 
\begin{equation}
\label{coeffRK1_2}
a_{k,\ell}(ih\omega_{p}/2)=\frac{1}{h} \int_{0}^{c_{k}h} e^{(c_{k}h-\sigma)i\omega_{p}/2}\mathcal{L}_{\ell}(\sigma) d\sigma,
\end{equation}
and
\begin{equation}
\label{coeffRK2_2}
b_{k}(ih\omega_{p}/2)=\frac{1}{h} \int_{0}^{h} e^{(h-\sigma)i\omega_{p}/2}\mathcal{L}_{k}(\sigma)d\sigma.
\end{equation}
We explain how one can compute these coefficients.
These functions $a_{k,\ell}$ and $b_k$ are holomorphic functions over $\C$.
After integration ({\it e.g.} using integration by parts)
over $(0,h)$ of the integrands defining these operators,
we obtain holomorphic functions (say, of {\color{black}$h\omega_p$}) with a removable
singularity at the origin of the complex plane.
For example, when the method has $s=2$ stages, the coefficient $a_{1,1}$ reads
\begin{equation}
\label{formulacoeff}
  a_{1,1}(ih\omega_{p}/2)=c_1^2\frac{e^{i c_1 h \frac{\omega_{p}}{2}}(1-i c_2h\frac{\omega_{p}}{2})-1+i h \frac{\omega_{p,q}}{2}(c_2-c_1)}{(c_1-c_2)((i c_1 h \omega_{p}/2)^2)}.
\end{equation}
{\color{black}
Applying directly a formula such as \eqref{formulacoeff} yields
numerical instabilities in the computation of the ratio when
$h|\omega_{p}|\ll 1$.
Therefore, we use two different strategies for the computation
of the coefficients \eqref{coeffRK1_2} and \eqref{coeffRK2_2}.
When $h|\omega_{p}|\leq 1/2$, we use a discretized version
of the Cauchy representation formula for a holomorphic function
$f$
\begin{equation}
\label{Cauchyintegral}
  f(z)
= \frac{1}{2i\pi} \int_{\mathcal C} \frac{f(\omega)}{\omega-z} {\rm d}\omega,
\end{equation}
(where $\mathcal C$ is the positively oriented unit circle and $|z|<1$),
following the method presented in \cite{KaTr2005, ScTr2007}.
When $h|\omega_{p}|>1/2$, we simply evaluate directly formulas of the
form \eqref{formulacoeff}.
}

\begin{remark}
  For the evaluation of $f(ih\omega_{p})$ when $h|\omega_{p}|\leq 1/2$,
we use the trapezoidal quadrature rule to compute an approximation
of the integral in \eqref{Cauchyintegral}.
Since the unit circle is a smooth path in the complex plane,
the function $t\mapsto i e^{it} f(e^{it})/(e^{it}-ih\omega_{p})$ is
a smooth $2\pi$-periodic function and hence the trapezoidal rule
has infinite order.
\end{remark}

One can in fact compute these $s\times (s+1)\times K^2$ coefficients
independently and use parallel computing.
Once these coefficients are computed, we can begin the time-stepping method
and no additional computation is required provided the discretization parameters $h$ and $K$ are fixed.

\subsection{Result}

For the solution of the Cauchy problem \eqref{Cauchy1}, using
an ERK method defined by the $s$ points
$0\leq c_1<\dots<c_s\leq 1$, we have the following
\begin{theorem}
\label{th:ordreERK}
  For all $\psi_0\in H^\sigma(\T_\delta^d)$ and all $T>0$
such that the exact solution of \eqref{Cauchy1} is smooth over $[0,T]$,
there exists $C,h_0>0$ such that for all $h\in(0,h_0)$,
the ERK method defined by
\eqref{nonlinsys}-\eqref{Nexttime} starting from $\psi_0$ is well-defined.
Moreover, we have for all $h\in(0,h_0)$ and $n\in\N$ such that $nh\leq T$,
\begin{equation*}
  \|\psi(t_n)-\psi_n\|_{H^\sigma(\T_\delta^d)} \leq C h^s
\end{equation*}
\end{theorem}

\begin{proof}
  This is a consequence of Theorem $3.6$ in \cite{Du2009}.
Let us check the hypotheses of Theorem 3.6. Hypothesis 3.1 is
straightforward from \eqref{eq:defNper}. Hypothesis 3.2 follows from
Lemma \ref{lemmaLipschitz} and the fact that $H^\sigma(\T_\delta^d)$ is
an algebra for $\sigma>d/2$. Hypotheses 3.3 and 3.4 follow from our
assumption on the temporal smoothness of the exact solution.
\end{proof}

 \begin{remark}
{\color{black}
As mentioned in the introduction, the temporal constant $C$ in
Theorem \ref{th:ordreERK} is \emph{not} uniform in the truncation
parameter $\delta$.
}
Note that our proof extends to general nonlinear Schr\"odinger
equations with a smooth time-dependant potential on the whole space $\R^d$
as soon as $\sigma>d/2$ and the multiplication by $V$ is a continuous
operator from $H^{\sigma}$ to itself.
 \end{remark}

\begin{remark}
  Even if we are not able to prove it, we observe in the numerical experiments that if the $s$ collocation points of the ERK method are Gauss points, the method is of numerical order $2s$ (see subsection \ref{Comp1d}).
  This indeed may be true when the methods are applied to the problem \eqref{Cauchy1} on the torus $\T_\delta^d$ because of the periodic boundary conditions while it may {\it not} be true when the methods are applied to the problem \eqref{GPeqnv2} set on the whole space $\R^d$. This is the case for example when one considers such methods applied to semilinear parabolic problems \cite{HoOs2005}.
This would indeed be another limitation of the reduction of the problem from $\R^d$ to $\T^d_\delta$ (see subsection \ref{subsec:simplifiedsetting}). 
In any case, the proof of superconvergence for Gauss-ERK methods applied to
problem \eqref{Cauchy1} on the torus $\T_\delta^d$ does not seem to be a
straightforward adaptation of the classical result for ODEs.
An option to prove it could be to use appropriate Taylor expansions of the exact
solution of the problem, but the way one can order the terms appearing in
the consistency error to obtain order $2s$ is really non trivial.
\end{remark}

{\color{black}  In the following we call Gauss-ERK method an ERK method using Gauss collocation points.}

\section{Lawson methods}
\label{Lawsonmethods}

In \cite{La1967}, Lawson considers the problem of designing
some Runge--Kutta type methods for stiff ordinary differential equations.
The idea is to perform a change of unknowns to transform the stiff system
into a related nonstiff one.
Then some basic Runge--Kutta method is applied to the related problem.
The combination is termed a generalized Runge--Kutta method in \cite{La1967}
and often called Lawson method. 

The goal of this section is to describe the implementation of such methods
on the problem \eqref{Cauchy1}
seen as an ordinary differential equation in time, to perform
an analysis of the order of convergence of these methods as well
as some of their preservation properties.

\subsection{Notations and description of the method}

Let us consider the equation \eqref{Cauchy1},
and set the following change of unknowns (so-called Lawson transformation),
\begin{equation}
\label{Lawsontransf}
u(t,{\mathbf{x}})=e^{-Lt}\psi(t,{\mathbf{x}}).
\end{equation}
Then $\psi$ solves \eqref{Cauchy1} if and only if
$u$ solves
\begin{equation}
\label{Cauchy2}
\begin{array}{ll}
\partial_{t} u(t,{\mathbf{x}})=e^{-Lt}N_w(t,e^{Lt}u({t,\mathbf{x}})), \quad (t,{\mathbf{x}}) \in [0,T] \times \T_\delta^d  \\[2mm]
u(0,{\mathbf{x}})=\psi(0,{\mathbf{x}})=\psi_{0}({\mathbf{x}}), \quad {\mathbf{x}} \in \T_\delta^d.
\end{array}
\end{equation}  
Now one can apply a classical Runge--Kutta method to \eqref{Cauchy2} seen as
an ordinary differential equation in time.
Assume $(a_{k,l})_{1\leq k,l\leq s}$ and $(b_k)_{1\leq k\leq s}$
are {\color{black}a matrix and a vector of real entries. }
Set 
\begin{equation}
\label{defci}
  \forall k\in\{1,\dots,s\},\qquad
  c_k:=\sum_{\ell=1}^s a_{k,\ell},
\end{equation}
and consider the $s$-stage classical
Runge--Kutta method with Butcher tableau given by
\begin{equation*}
  \begin{array}{c|ccc}
    c_1 & a_{1,1} & \cdots & a_{1,s} \\
    \vdots & \vdots & & \vdots \\
c_s & a_{s,1} &\cdots & a_{s,s} \\ \hline
 & b_1 & \cdots & b_s
  \end{array}
\end{equation*}
Assume that this Runge-Kutta method is of order at least $1$.
This means that the following condition is fulfilled:
\begin{equation}
\label{condRK1}
  \sum_{k=1}^s b_{k}=1.
\end{equation}
Applying this Runge--Kutta method to the problem \eqref{Cauchy2}
defines a Lawson method:
we compute an approximation $u_{n+1}$
at time $t_{n+1}$ of the exact solution from an approximation $u_n$ at time $t_n$
by solving the system of $s$ nonlinear equations (the unknowns
being the $(u_{n,k})_{1\leq k\leq s}$):
\begin{equation}
\label{Lawsonv}
  u_{n,k} = u_n + h \sum_{\ell=1}^s a_{k,\ell} e^{-(t_n+c_\ell h) L}
N_w\left(t_n+c_\ell h,e^{(t_n+c_\ell h) L} u_{n,\ell}\right),
\end{equation}
and then we compute $u_{n+1}$ through the formula
\begin{equation}
\label{Lawsonv2}
  u_{n+1} = u_n + h \sum_{k=1}^s b_k e^{-(t_n+c_k h) L} 
N_w\left(t_n+c_k h,e^{(t_n+c_k h) L} u_{n,k}\right).
\end{equation}
Equivalently, the Lawson method for the unknowns
\begin{equation}
\label{LawsonPsi1}
  \psi_{n,i}:=e^{(t_n+c_i h)L}u_{n,i}
\qquad {\rm and}
\qquad
  \psi_n:= e^{t_n L} u_n,
\end{equation}
consists in solving the $s$ nonlinear equations
\begin{equation}
\label{LawsonPsi2}
  \psi_{n,k} = e^{c_k h L} \psi_n + h \sum_{\ell=1}^s a_{k,\ell} e^{(c_k-c_\ell) h L}
N_w\left(t_n+c_\ell h, \psi_{n,\ell}\right),
\end{equation}
and then computing $\psi_{n+1}$ through the formula
\begin{equation}
\label{LawsonPsi3}
  \psi_{n+1} = e^{hL} \psi_n + h \sum_{k=1}^s b_k e^{(1-c_k) h L}
N_w\left(t_n+c_k h, \psi_{n,k}\right).
\end{equation}
We simply denote these relations \eqref{LawsonPsi2}-\eqref{LawsonPsi3} by
$\psi_{n+1}=\Phi_{t_{n}\rightarrow t_{n+1}}(\psi_{n})$.
As for the ERK methods of \eqref{expRK},
the Runge--Kutta method defined by $a_{k,\ell}, b_{k}$
is referred to as the underlying Runge--Kutta method. 

Note that, in view of Lemma \ref{lemmaLipschitz},
the Lawson method \eqref{LawsonPsi1}-\eqref{LawsonPsi3}
is well defined in $H^\sigma(\T_\delta^d)$
provided $h>0$ is small enough.

\subsection{Results}

In this section, we present some results on the Lawson method given by \eqref{LawsonPsi2}-\eqref{LawsonPsi3}. First of all, since equation \eqref{Cauchy1} is time reversible, we give a sufficient condition for the Lawson method to be symmetric. We follow ideas developed in \cite{CeCoOw2008}, where the authors are solving an autonomous NLS equation, and we show that although equation \eqref{Cauchy1} is non autonomous, the sufficient condition is quite similar.

\begin{theorem} \label{thsym}
  Assume that the $s$-stage Runge--Kutta method defined by \mbox{}
$(a_{k,\ell})_{1\leq k,\ell\leq s}$ and $(b_k)_{1\leq k\leq s}$ satisfies
\eqref{condRK1} so
that it is of order at least 1.
Assume that this method satisfies
\begin{equation}
  \label{Condsymetrie}
  \forall (k,\ell)\in\{1,\dots,s\}^2,\qquad
  a_{s+1-k,s+1-\ell}+a_{k,\ell}=b_\ell,
\end{equation}
so that it is symmetric (see Theorem 2.3 in \cite{HLW}).
Then the Lawson method defined by \eqref{LawsonPsi2}-\eqref{LawsonPsi3} is
also symmetric.
\end{theorem}

\begin{proof}
First of all, let us mention that the symmetry condition
\eqref{Condsymetrie} gives
\begin{equation}
\label{Condb_l}
b_{\ell}=b_{s+1-\ell}, \quad 1 \leq \ell \leq s.
\end{equation}
Moreover, summing \eqref{Condsymetrie} over $\ell$ and
using \eqref{condRK1} we have
\begin{equation}
\label{Condc_k}
1-c_{k}=c_{s+1-k}, \quad 1\leq k \leq s.
\end{equation}
The adjoint $\Phi_{t_{n}\rightarrow t_{n+1}}^{\star}$ of the method $\Phi_{t_{n}\rightarrow t_{n+1}}$ is by definition $\Phi_{t_{n+1}\rightarrow t_{n}}^{-1}$. The relation $\hat{\psi}_{n+1}=\Phi_{t_{n}\rightarrow t_{n+1}}^{\star}(\hat{\psi}_{n})$ is equivalent to $\hat{\psi}_{n}=\Phi_{t_{n+1}\rightarrow t_{n}}(\hat{\psi}_{n+1})$. This corresponds to exchanging $t_{n}$ with $t_{n+1}$ and $h$ with $-h$ in  \eqref{LawsonPsi2}-\eqref{LawsonPsi3}. This leads to
\begin{equation}
\label{LawsonPsi2bis}
  \hat{\psi}_{n,k} = e^{-c_k h L} \hat{\psi}_{n+1} - h \sum_{\ell=1}^s a_{k,\ell} e^{-(c_k-c_\ell) h L}
N_w\left(t_{n+1}-c_\ell h,\hat{\psi}_{n,\ell}\right),
\end{equation}
\begin{equation}
\label{LawsonPsi3bis}
  \hat{\psi}_{n} = e^{-hL} \hat{\psi}_{n+1} - h \sum_{k=1}^s b_k e^{-(1-c_k) h L}
N_w\left(t_{n+1}-c_k h, \hat{\psi}_{n,k}\right).
\end{equation}
Extracting $\hat\psi_{n+1}$ from \eqref{LawsonPsi3bis} gives
\begin{equation}
\label{vianney}
  \hat{\psi}_{n+1} = e^{hL} \hat{\psi}_{n} + h \sum_{k=1}^s b_k e^{c_k h L}
N_{w}\left(t_{n+1}-c_k h, \hat{\psi}_{n,k}\right).
\end{equation}
Plugging this expression into \eqref{LawsonPsi2bis}, we get
\begin{equation*}
  \hat{\psi}_{n,k} = e^{(1-c_k) h L} \hat{\psi}_{n} + h \sum_{\ell=1}^s (b_\ell-a_{k,\ell}) e^{-(c_{k}-c_\ell) h L}
N_w\left(t_{n+1}-c_\ell h, \hat{\psi}_{n,\ell}\right).
\end{equation*}
Using \eqref{Condsymetrie}, we infer
\begin{equation}
\label{vianney2}
  \hat{\psi}_{n,k} = e^{(1-c_k) h L} \hat{\psi}_{n} + h \sum_{\ell=1}^s a_{s+1-k,s+1-\ell} e^{-(c_{k}-c_\ell) h L}
N_w\left(t_{n+1}-c_\ell h, \hat{\psi}_{n,\ell}\right).
\end{equation}
Setting $\tilde\psi_{n,k}=\hat\psi_{n,s+1-k}$ and reordering the sums in \eqref{vianney}-\eqref{vianney2}, we obtain
\begin{equation*}
  \tilde{\psi}_{n,k} = e^{(1-c_{s+1-k}) h L} \hat{\psi}_{n} + h \sum_{\ell=1}^s a_{k,\ell} e^{-(c_{s+1-k}-c_{s+1-\ell}) h L}
N_w\left(t_{n}+(1-c_{s+1-\ell}) h, \tilde{\psi}_{n,\ell}\right),
\end{equation*}
\begin{equation*}
  \hat{\psi}_{n+1} = e^{hL} \hat{\psi}_{n} + h \sum_{k=1}^s b_{s+1-k} e^{c_{s+1-k} h L}
N_w\left(t_{n}+(1-c_{s+1-k} )h, \tilde{\psi}_{n,k}\right).
\end{equation*}
Using \eqref{Condb_l}-\eqref{Condc_k} we conclude that
\begin{equation*}
  \tilde{\psi}_{n,k} = e^{c_{k} h L} \hat{\psi}_{n} + h \sum_{\ell=1}^s a_{k,\ell} e^{(c_{k}-c_{\ell}) h L}
N_w\left(t_{n}+c_{\ell} h, \tilde{\psi}_{n,\ell}\right),
\end{equation*}
\begin{equation*}
  \hat{\psi}_{n+1} = e^{hL} \hat{\psi}_{n} + h \sum_{k=1}^s b_{k} e^{(1-c_{k}) h L}
N_w\left(t_{n}+c_{k} h, \tilde{\psi}_{n,k}\right).
\end{equation*}
This proves that the Lawson method is symmetric. 
\end{proof}

If the underlying Runge--Kutta method preserves quadratic invariants in the sense that it satisfies the Cooper condition, then so does the associated Lawson method.

\begin{theorem}\label{quadinv}
Assume that the underlying Runge-Kutta method satisfies \eqref{condRK1} so
that it is of order at least 1. Assume that it satisfies the Cooper condition,
\begin{equation}
\label{CondCooper}
b_{k}a_{k,\ell}+b_{\ell}a_{\ell,k}=b_{k}b_{\ell}, \quad \forall ~1 \leq k, \ell \leq s,
\end{equation}
so that it preserves quadratic invariants. Then the Lawson method \eqref{LawsonPsi2}-\eqref{LawsonPsi3} preserves the $L^{2}$-norm:
\begin{equation*}
\|\psi_{n}\|_{L^{2}(\T_\delta^d)}=\|\psi_{0}\|_{L^{2}(\T_\delta^d)},
\quad \forall ~n \geq 0.
\end{equation*} 
\end{theorem}

\begin{proof}
Since the evolution equation \eqref{Cauchy1} preserves the $L^{2}$-norm, and so does the change of variables \eqref{Lawsontransf}, the evolution equation \eqref{Cauchy2} also preserves the $L^{2}$-norm. Therefore $t \mapsto \int_{\T_\delta^d} |\psi(t,\mathbf{x})|^{2} {\rm d} \mathbf{x}$ is a quadratic invariant of equation \eqref{Cauchy2}. The Cooper condition guarantees that the Runge--Kutta method \eqref{Lawsonv}-\eqref{Lawsonv2} preserves quadratic invariants. Hence for all $n$, $\int_{\T_\delta^d} |v_{n}(\mathbf{x})|^{2} {\rm d} \mathbf{x}=\int_{\T_\delta^d} |v_{0}(\mathbf{x})|^{2} {\rm d} \mathbf{x}$ and then $\int_{\T_\delta^d} |\psi_{n}(\mathbf{x})|^{2} {\rm d} \mathbf{x}=\int_{\T_\delta^d} |\psi_{0}(\mathbf{x})|^{2} {\rm d} \mathbf{x}$.
\end{proof}

\begin{definition}
\label{DefGaussLawson}
We call Gauss-Lawson method any Lawson method of the form \eqref{LawsonPsi2}-\eqref{LawsonPsi3} such that the underlying Runge--Kutta method is a Gauss collocation method (see section II.1.3 in \cite{HLW}). 
\end{definition}

\begin{corollary}
A Gauss-Lawson method applied to \eqref{Cauchy1} is symmetric and preserves the $L^{2}$-norm.
\end{corollary}

\begin{proof}
The underlying Runge--Kutta method is a collocation method at Gauss points and therefore it is symmetric (see Corollary 2.2, Chapter V in \cite{HLW}). Hence the Gauss-Lawson method applied to \eqref{Cauchy1} is symmetric by Theorem \ref{thsym}.  Moreover the underlying Runge--Kutta method satisfies the Cooper condition \eqref{CondCooper} (see exercise 5 in chapter IV in \cite{HLW}). Therefore the Gauss-Lawson method preserves the $L^{2}$-norm by \eqref{quadinv}.
\end{proof}

We now want to prove that the Gauss-Lawson method with $s$ stages applied
to \eqref{Cauchy1} has order $2s$ in $H^\sigma(\T_\delta^d)$ for $\sigma>d/2$
(see Theorem \ref{th:ordreGaussLawson}).
Our strategy is the following.
First, we consider an equivalent autonomous form of \eqref{Cauchy1}.
Second,we show that applying a Lawson method to this autonomous form
is essentially the same as applying the method to
\eqref{Cauchy1} directly.
Third, we rely on an Alekseev-Gr\"obner lemma for autonomous systems
which provides a representation of the error that allows {\color{black}us} to conclude that,
since Gauss-Lawson methods are collocation methods,
they have order $2s$.

Let us set
\begin{equation*}
  U(t)=
  \begin{pmatrix}
    t\\
    u(t)
  \end{pmatrix}
  \in\R\times H^\sigma(\T_\delta^d),
\end{equation*}
so that \eqref{Cauchy1} reads
\begin{equation}
\label{eq:autonomous}
  \frac{{\rm d}}{{\rm d} t}
  U (t) = 
  F(U(t)),
\end{equation}
with
\begin{equation}
\label{eq:defF}
  F
  \begin{pmatrix}
    t\\
    u
  \end{pmatrix}
  =
    \begin{pmatrix}
      1\\
      e^{-tL}N_{w}(t,e^{tL}u)
    \end{pmatrix}.
\end{equation}
We have the following useful
\begin{lemma}
\label{uU}
  Let us fix $u_0\in H^\sigma(\T_\delta^d)$. For all $h>0$ sufficiently small,
we apply the Lawson method \eqref{Lawsonv}-\eqref{Lawsonv2} and denote
by $u_n$ the corresponding numerical values. Similarly,
we start with $U_0=(0,u_0)^t$, we apply the same method to the Cauchy problem
\eqref{eq:autonomous} with intial datum $U(0)=U_0$ and we denote by
$U_n$ the corresponding numerical values. We have for all $n\in\N$ such that
$nh\leq T$,
\begin{equation*}
  U_n=
  \begin{pmatrix}
    nh\\
    u_n
  \end{pmatrix}.
\end{equation*}
\end{lemma}

\begin{proof}
  Since the function $N_w$ satisfies a local Lipschitz condition
(see Lemma \ref{lemmaLipschitz}), so does $F$ (see \eqref{eq:defF}).
Hence, the Lawson methods are well defined locally for $h>0$ small enough.
We perform the proof by induction. The relation is true for $n=0$.
Assume it holds for some $n\in\N$ such that $(n+1)h\leq T$.
The definition of the coefficients $(c_k)_{1\leq k\leq s}$ (see \eqref{defci})
and the first component of the function $F$
ensure that for all $k\in\{1,\dots,s\}$, the first coefficient of $U_{n,k}$
is $t_{n,k}=nh+c_k h$. Therefore, in view of \eqref{Lawsonv},
the $(U_{n,k})_{1\leq k\leq s}=(nh+c_k h,u_{n,k})^t$  are the unique local
solutions for the Lawson inner problem in $U$.
Similarly, the consistency relation \eqref{condRK1} ensures that
the first component of $U_{n+1}$ is $(n+1)h$. Hence, in view of \eqref{Lawsonv2},
we have that the second component of $U_{n+1}$ is actually $u_{n+1}$.
\end{proof}

We are now able to state
\begin{theorem}
\label{th:ordreGaussLawson}
  Assume $u_0\in H^\sigma(\T_\delta^d)$ {\color{black} is given}
and $T>0$ {\color{black}is chosen such that the exact solution
$t\mapsto u(t)$ of the Cauchy problem \eqref{Cauchy2} is defined over $[0,T]$}.
There exist constants $C>0$ and $h_0>0$ such that
the corresponding numerical approximations provided by the Gauss-Lawson
method with $s$ stages $u_n$ satisfy
\begin{equation*}
  \forall h\in(0,h_0),\ \forall n\in\N\text{ s.t. } 0\leq nh\leq T,\qquad
  \|u(t_n)-u_n\|_{H^\sigma(\T_\delta^d)} \leq C h^{2 s}.
\end{equation*}
\end{theorem}

{\color{black}
\begin{remark}
This implies that the exact solution $\psi$ of \eqref{Cauchy1} also satisfies
\begin{equation*}
  \forall h\in(0,h_0),\ \forall n\in\N\text{ s.t. } 0\leq nh\leq T,\qquad
  \|\psi(t_n)-\psi_n\|_{H^\sigma(\T_\delta^d)} \leq C h^{2 s},
\end{equation*}
where $\psi_n$ are the numerical approximations of $\psi$ by the Gauss-Lawson method \eqref{LawsonPsi2}-\eqref{LawsonPsi3}.
\end{remark}
}

{\color{black}
\begin{remark}
As mentioned in the introduction, the temporal constant $C$ in
Theorem \ref{th:ordreGaussLawson} is \emph{not} uniform in the truncation
parameter $\delta$.
\end{remark}
}

\begin{proof}
  Thanks to Lemma \ref{uU}, it is sufficient to prove
that the Gauss-Lawson method applied to the Cauchy problem
\eqref{eq:autonomous} with initial datum $U_0=(0,u_0)$ is of order $2s$.
We chose $h_0>0$ sufficiently small to ensure that the method is
well-defined for all $h\in(0,h_0)$.
We follow the proof of {\color{black}\cite{HNW93} and} \cite{Is2009}
for collocation methods applied to ODEs.
{\color{black} Indeed, with our hypotheses ($\sigma>d/2, \kappa\in\N$)
the vector field $F$ defined in \eqref{eq:defF}
is
$\mathcal C^{\infty}(\R\times H^\sigma(\T_\delta^d),\R\times H^\sigma(\T_\delta^d))$
and we are dealing with an autonomous ODE in the Banach space
$\R\times H^\sigma(\T_\delta^d)$ with smooth vector field.
Hence, solutions of \eqref{eq:autonomous} are smooth functions with values
in $\R\times H^\sigma(\T_\delta^d)$.
Therefore, the variational evolution equation for
$Y=\frac{\partial U}{\partial U_0}$:
\begin{equation*}
  Y'(t)=F'(U(t))Y(t),
\end{equation*}
with $Y(0)=(1,{\rm Id_{H^\sigma(\T_\delta^d)}})$ has a right-hand side linear
operator $t\mapsto F'(U(t)) \in \mathcal C^\infty(I,\mathcal L(\R\times H^\sigma(\T^d_\delta)))$, where $I$ is an interval where $U$ is defined and $\mathcal L(\R\times H^\sigma(\T^d_\delta))$ stands for the space of linear continuous operators
of $\R\times H^\sigma(\T^d_\delta)$.
}
First, we recall that the mapping $U_n\mapsto U_{n+1}$ can be defined
through a collocation problem. The same problem allows {\color{black}us}
to express
the consistency error: find $P$ a polynomial of degree at most $s$
(with coefficients in $\R\times H^\sigma(\T_\delta^d)$) such that
\begin{equation*}
 \left\{
  \begin{array}{l}
  P(t_n)=U(t_n), \\ 
  P'(t_n+c_k h) = F(P(t_n+c_kh)), \qquad \qquad \text{for } 1\leq k \leq s,
  \end{array}
  \right.
\end{equation*}
and set $\tilde U_{n+1}=P(t_{n+1})$, so that the consistency error
reads $\tilde U_{n+1}-U(t_{n+1})$.
Second, define the {\color{black} defect} $\tilde d$ as
\begin{equation*}
  \tilde d(t,P(t))=P'(t)-F(P(t)).
\end{equation*}
Then, we use the Alekseev-Gr\"obner lemma (Theorem 2 of \cite{DeTa2013}),
which ensures that there exists
a smooth function $g$ (corresponding to $\partial_2{\mathcal E}_G$
in \cite{DeTa2013}) from $[0,T]\times(\R\times H^\sigma(\T_\delta^d))$ to
the linear continuous operators of $\R\times H^\sigma(\T_\delta^d)$ such that
the following identity holds in $H^\sigma(\T_\delta^d)$: for all $t\in [0,h]$,
\begin{equation*}
  P(t_n+t)-U(t_n+t) = \int_{t_n}^{t_n+t} g(t_n+t-\tau,P(\tau)) \tilde
d(\tau,P(\tau))
{\rm d} \tau.
\end{equation*}
Setting $t=h$ in this formula and using the fact that $g$ is smooth
and all the derivatives of $t\mapsto \tilde d(t,P(t))$ are bounded on $[0,h]$
independently of $h$ (see Lemma 1.6 of Chapter 2 in \cite{HLW}), we infer
that the consistency error in $H^\sigma(\T_\delta^d)$ is of order $h^{2s+1}$.
Hence the global error is of order $h^{2s}$ via a classical discrete Gronwall
lemma.
\end{proof}

\section{Numerical experiments}
\label{Comparaisons}

In this section, we make some numerical experiments and show the numerical
efficiency of the methods.
First, we consider one dimensional problems and compare some ERK methods
and Gauss-Lawson methods with the classical splitting methods.
Second, we compare the same kinds of methods applied to an actual 2D problem
with a rotating BEC. 
{\color{black}As expected in the theoretical results above 
, the numerical behavior does not actually depend much on the dimension as is illustrated in subsections \ref{Comp1d} and \ref{Comp2d}.} 

\subsection{One dimensional example: \textcolor{black}{the one-dimensional cubic NLS equation}}
\label{Comp1d}

We provide in this subsection some numerical experiments to show the efficiency
of the Gauss-ERK and the Gauss-Lawson methods compared with the traditional
splitting methods.
We use here splitting methods respectively of order 1, 2, 4 and 6
\cite{blanes_2002}.
In order to present the usual splitting schemes, we have to define the
operators $S_L$ and $S_{N_w}$ associated respectively to the evolution of the
{\color{black}equations}
$$
\partial_t u(t,\mathbf{x})=Lu(t,\mathbf{x}) \quad \partial_t v(t,\mathbf{x})= N_w v(t,\mathbf{x}),
$$
where $L=i\Delta/2$ and $N_w$ is defined in \eqref{eq:defNper}. The operators
satisfy by definition the following relations involving the exact solutions of the associated {\color{black}equations}:
\begin{equation*} 
u(t+h,\mathbf{x})=S_L(h) u(t,\mathbf{x})
\qquad {\rm and}\qquad
v(t+h,\mathbf{x})=S_{N_w}(h) v(t,\mathbf{x}).
\end{equation*}
As explained for example in \cite{blanes_2002},
a splitting idea for building a splitting scheme of even order consists
in approximating the continous flow associated to
$\partial_t \zeta(t,\mathbf{x})=L\zeta(t,\mathbf{x})+N_w(t,\zeta(t,\mathbf{x}))$
by a composition of operators $S_L$ and $S_{N_w}$ of the form
\begin{equation*}
S_{N_w}(a_1 h)S_L(b_1 h) \cdots S_{N_w}(a_r h)S_L(b_r h)S_{N_w}(a_{r+1}h),
\end{equation*}
with $r\geq 1$ and for all $\ell$,
$a_{r+2-\ell}=a_\ell$ and $b_{r+1-\ell}=b_\ell$.
The coefficients of the splitting methods we implemented
are given in Table \ref{tab:split}.
\begin{table}[!h]
  \centering
  \begin{tabular}{|l|l|l|}\hline
    Order 2   & Order 4                         & Order 6 \\ \hline
    $a_1=1/2$ & $a_1=\theta$                    &  $a_1 = 0.0502627644003922$      \\
    $b_1=1$   & $a_2=1/2-\theta$                &  $a_2 = 0.413514300428344$       \\
              & $b_1=2\theta$                   &  $a_3 = 0.0450798897943977$      \\
              & $b_2=1-4\theta $                &  $a_4 = - 0.188054853819569$     \\
              &  $\theta=(2+2^{1/3}+2^{-1/3})/6$ & $a_5 = 0.541960678450780$            \\
              &                                &   $a_6 =1-2(a_1+a_2+a_3+a_4+a_5)$ \\[1mm] 
& &   $b_1 = 0.148816447901042$      \\
&&    $b_2 = - 0.132385865767784$     \\
&&    $b_3 = 0.067307604692185$       \\
&&   $b_4 = 0.432666402578175$        \\
& &      $b_5 = 0.5 -(b_1+b_2+b_3+b_4)$\\ \hline
  \end{tabular}
  \caption{Coefficient of splitting methods.}
  \label{tab:split}
\end{table}
Let us mention that alternative ways of deriving splitting methods for such
problems exist, some of which authorize the use of complex coefficients. {\color{black}For example, one can use splitting methods such as (5.1) in \cite{CaChDeVi2009}, where only one set of coefficients has nonzero imaginary part and the other set can be used to solve the linear part of a Schr\"odinger equation (see also \cite{blanes_2013}).}

We compute the numerical solution to the one-dimensional cubic NLS equation
\begin{equation}
  \label{eq:schro_num}
  \partial_t \psi = i \partial_x^2 \psi + i q |\psi|^2 \psi, \quad (t,x) \in [0,T]\times \R.
\end{equation}
An exact solution for $(t,x) \in [0,T]\times\R$ is given by the soliton formula
\begin{equation}
\label{eq:psiex}
\psi_{ex}(t,x)=\frac{2a}{q} \text{sech} \big ( (x-x_0)-c t\big ) \exp \Big ( ic \frac{(x-x_0)-ct}{2} \Big ) \exp \Big ( i \big (a+\frac{c^2}{4} \big) t \Big ).
\end{equation}
For the numerical approximation of this solution, as explained in
subsection \ref{subsec:simplifiedsetting}, we take a periodic finite interval
$(x_\ell,x_r)$
of big enough length, and we discretize the space operators using Fourier
spectral approximation.
We choose the spatial mesh size $k=\Delta x >0$ with $k=(x_r-x_\ell)/M$ with $M=2^P$, $P\in \N^*$. The time step is denoted by $h=\Delta t$ and $h=T/N_T$ for some $N_T\in\N^\star$. The grid points and the discrete times are
$$
x_j:=x_\ell+jk, \quad t_n:=nh, \quad j=0,1,\cdots,M, \quad n=0,1,\cdots,N_T.
$$
Let $\psi_j^n$ be the approximation of $\psi(t_n,x_j)$. Since we discretize \eqref{eq:schro_num} by the Fourier spectral method, $\psi_j^n$ and
its Fourier transform satisfy the following relations:
$$
\psi_j^n = \frac{1}{M} \sum_{m=-M/2}^{M/2-1} \hat{\psi}_m^n e^{i\mu_m(x_j-x_\ell)}, \quad j=0,\cdots,M-1,
$$
and
$$
\hat{\psi}_m^n=\sum_{j=0}^{M-1} \psi_j^n e^{-i\mu_m(x_j-x_\ell)}, \quad m=-\frac{M}{2},\cdots,\frac{M}{2}-1,
$$
where $\mu_m=\frac{2\pi m}{x_r-x_\ell}$ for all $m=-\frac{M}{2},\cdots,\frac{M}{2}-1$.
Let us define the discrete gradient operator $\nabla_k$
$$
\widehat{(\nabla_k v)}_m = i \mu_m  \hat{v}_m, \quad v\in \C^M.
$$ 
Let us denote by $\Pi_k$ the projection operator
$$
\begin{array}{llcl}
\Pi_k: & \mathcal{C}^0([x_\ell,x_r],\C) & \to & \C^{M}\\
       & \varphi & \mapsto & \left (\varphi(x_j) \right )_{0\leq j \leq M-1}  
\end{array}.
$$
We define the discrete $\ell^r$ norm on $\C^M$ as
$$
\|v\|_{\ell^r} = \left (k \sum_{j=0}^{M-1} |v_j|^r\right)^{1/r}, \quad v\in \C^M, \ 
r\geq 1.
$$
Using these definitions, we consider the following errors:
\begin{equation}
  \label{eq:errphase}
  \mathcal{E}_{P,h}=\sup_{n\in\{0,\cdots,N\}} \left \| \Pi_k(\psi_{ex}(t_n,\cdot))-(\psi_j^n)_j \right \|_{\ell^2},
\end{equation}
\begin{equation}
  \label{eq:errmass}
  \mathcal{E}_{M,h}=\sup_{n\in\{0,\cdots,N\}} \left |\left \| \Pi_k(\psi_{ex}(t_n,\cdot))\right \|_{\ell^2}-\left \|(\psi_j^n)_j \right \|_{\ell^2}\right |/\left \| \Pi_k(\psi_{ex}(0,\cdot))\right \|_{\ell^2}.
\end{equation}
We also define the discrete energy:
$$
E_{k}(v)=\frac{1}{2}\|\nabla_k v\|_{\ell^2}^2-\frac{q}{4} {\|v\|_{\ell^4}^4},
$$
the energy conservation is seen through the following relative error
\begin{equation}
  \label{eq:errener}
  \mathcal{E}_{E,h}=\sup_{n\in\{0,\cdots,N\}} \left |E_k(\Pi_k(\psi_{ex}(t_n,\cdot)))-E_k((\psi_j^n)_j)\right |/E_k(\Pi_k(\psi_{ex}(0,\cdot))).  
\end{equation}

For all the following simulations, we consider the computational domain $(x_\ell,x_r)=(-15,15)$ and the final time $T=5$. The experiments are performed with the discretization parameters $P=10$ and various time steps $h$ and the chosen physical parameters are $q=8$, $a=q^2/16$ and $c=0.5$. 
We provide for all methods the evolution of $\mathcal{E}_{P,h}$, $\mathcal{E}_{M,h}$ and $\mathcal{E}_{E,h}$ for various time steps $h$. All of our methods being implicit, it is necessary to solve a nonlinear problem at each time step for both ERK and Lawson method. This is performed here through a fixed point algorithm. It is therefore important to show the efficiency of the schemes to plot the evolution of the CPU time with respect to the error $\mathcal{E}_{P,h}$. \\

The legends of the figures are respectively given in Figures \ref{fig:legenda} and \ref{fig:legendb} for ERK and Lawson methods (we recall that $s$ denotes the number of collocation points) and splitting schemes. We showed in the previous sections that if one uses arbitrary distinct collocation points, the order of ERK and Lawson methods are $h^s$. Therefore, in all figures, curves associated to $s=1$, $2$ and $4$ respectively correspond to splitting of order $1$, $2$ and $4$. If we would like to compare the splitting of order 6 to ERK and Lawson methods, we would have to consider $s=6$ collocation points.
\newcommand{\goodgap}{%
  \hspace{\subfigtopskip}%
  \hspace{\subfigbottomskip}}

\begin{figure}[!h]
  \centering
  \subfigure[ERK and Lawson methods]{\footnotesize \input{legend}\label{fig:legenda}}  
  \hspace*{1cm}
  \subfigure[Splitting schemes]{\footnotesize \input{legend_split}\label{fig:legendb}}  
  \caption{Legends}
  \label{fig:legends}
\end{figure}
We gather the evolution of $\mathcal{E}_{P,h}$ for various schemes in Figure \ref{fig:EPh}. Note that since we use Gauss collocation points, we clearly obtain numerically an order $2s$ for both Gauss-ERK and Gauss-Lawson methods with $s$ stages. {\color{black}This is predicted} for Gauss-Lawson schemes thanks to Theorem \ref{th:ordreGaussLawson}. 
Let us remark that there is a saturation phenomenon {\color{black}for high $s$: the phase error $\mathcal{E}_{P,h}$ ceases to decrease when $h$ decreases. This is due to the fact that the exact solution $\psi_{ex}$ defined in \eqref{eq:psiex} does not satisfy exactly  the boundary conditions (see the transformation of the problem from $\R^d$ to $\T^d_\delta$ in section
\ref{subsec:simplifiedsetting}).
This can be reduced either by increasing the size
$\delta$ of the computational domain
or by changing the periodic boundary conditions to, for example,
homogeneous Dirichlet ones (in this case, one can replace FFT with
the discrete sine transform).}
If we compare Figures \ref{fig:EPhERK} and \ref{fig:EPhLaw}, it is noticeable that the constants of error are better for Gauss-ERK methods.

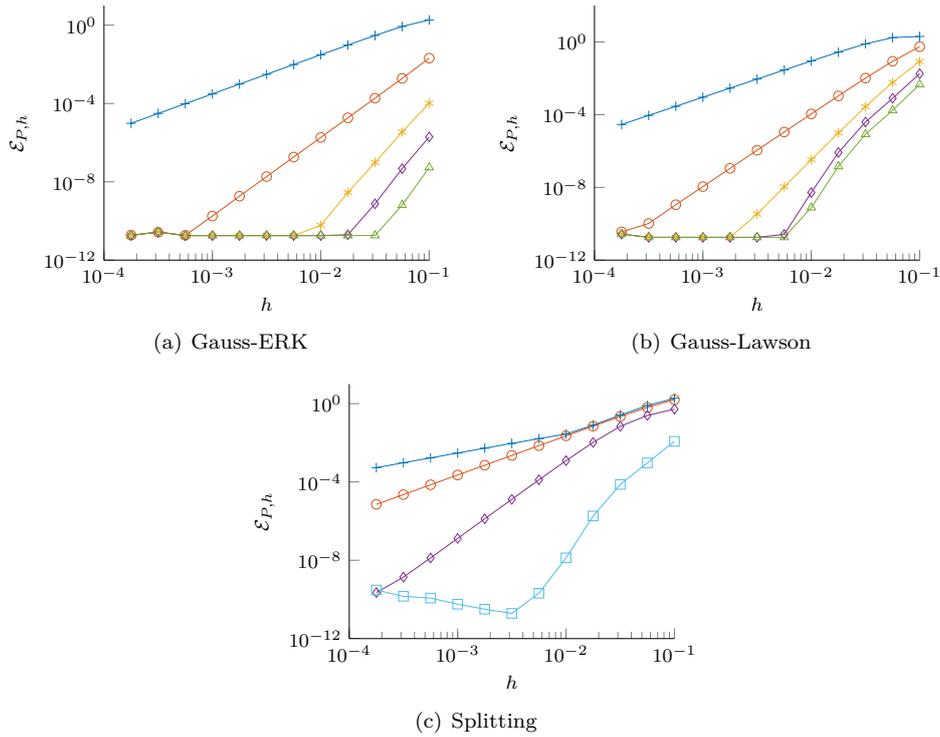
\begin{figure}[!h]
  \centering
  \subfigure[Gauss-ERK]{{\footnotesize \input{gauss_errphase_wrt_dt}}\label{fig:EPhERK}}
\goodgap  \subfigure[Gauss-Lawson]{{\footnotesize\input{lawson_errphase_wrt_dt}}\label{fig:EPhLaw}}\\
\subfigure[Splitting]{{\footnotesize\input{split_errphase_wrt_dt}}\label{fig:EPhSplit}}
  \caption{Evolution of $\mathcal{E}_{P,h}$ error with respect to the time step {\color{black}for problem \eqref{eq:schro_num}} {\color{black}(see Figure \ref{fig:legends} for legends)}}
  \label{fig:EPh}
\end{figure}

The preservations of mass and energy are fundamental when one deals
with dispersive equations.
{\color{black}
Let us first discuss mass preservation.
The figure presented in Figure \ref{fig:EMh} shows the evolution of
$\mathcal{E}_{M,h}$ with respect to the time step for several Gauss--ERK methods.
Indeed, Gauss-Lawson methods preserve mass up to round-off (see Theorem
\ref{quadinv}) and so do splitting methods.
It is noticeable, however, that, on the numerical example
\eqref{eq:schro_num}-\eqref{eq:psiex}
a Gauss-ERK method with at least two stages achieves round-off in mass
preservation for time steps below $10^{-2}$.
}

\begin{figure}[!h]
  \centering
\input{gauss_mass_wrt_dt}
  \caption{Evolution of $\mathcal{E}_{M,h}$ error with respect to the time step {\color{black}for problem \eqref{eq:schro_num}} using Gauss-ERK methods {\color{black}(see Figure \ref{fig:legends} for legends)}}
  \label{fig:EMh}
\end{figure}
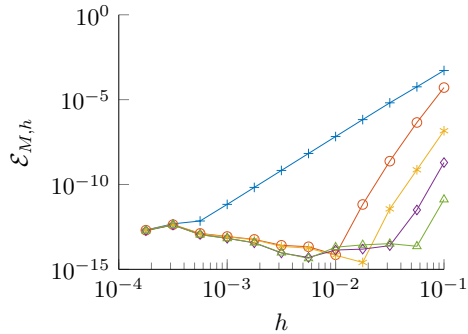

Concerning the energy conservation, none of the presented methods is able to handle it properly. However, it is clear in Figure \ref{fig:EEh}, which displays the evolution of $\mathcal{E}_{E,h}$ as a function of the time step for all schemes, that, with respect to energy preservation, Gauss-ERK methods are of better quality than Gauss-Lawson ones which are themselves better than splitting schemes.

\begin{figure}[!h]
  \centering
  \subfigure[Gauss-ERK]{{\footnotesize \input{gauss_ener_wrt_dt}}}
\goodgap  \subfigure[Gauss-Lawson]{{\footnotesize\input{lawson_ener_wrt_dt}}}\\
\subfigure[Splitting]{{\footnotesize\input{split_ener_wrt_dt}}}
  \caption{Evolution of $\mathcal{E}_{E,h}$ error with respect to the time step {\color{black}for problem \eqref{eq:schro_num}} {\color{black}(see Figure \ref{fig:legends} for legends)}}
  \label{fig:EEh}
\end{figure}
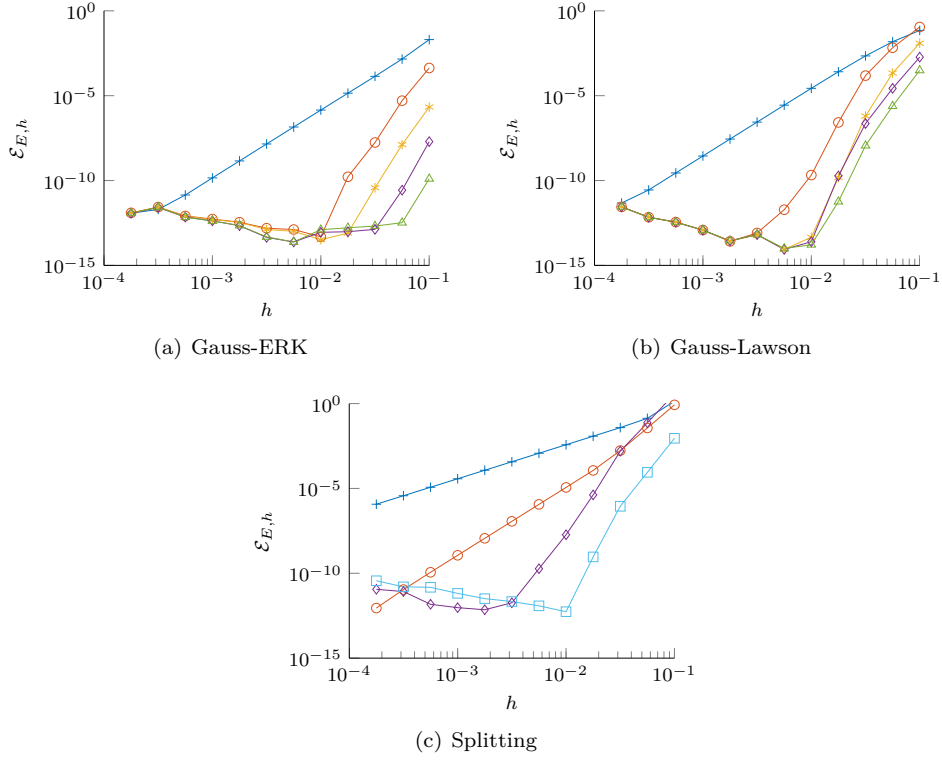

Finally, Figure \ref{fig:CPU} shows the evolution of the CPU time with respect to the $\mathcal{E}_{P,h}$ error. Theses figures are very interesting since one could think that implicit methods are more costly compared to splitting schemes. For a given error $\mathcal{E}_{P,h}$, the CPU time is clearly lower for Gauss-ERK schemes. For a ``big'' phase error $\mathcal{E}_{P,h} \geq 10^{-2}$, the splitting schemes are less costly than Gauss-Lawson ones. The situation is reversed when one is interested in ``small'' phase errors $\mathcal{E}_{P,h} \leq 10^{-2}$ and Gauss-Lawson schemes have to be recommended. For both ERK and Lawsons schemes, it is not necessarily interesting to use a high number of collocation points. 

\begin{figure}[!h]
  \centering
  \subfigure[Gauss-ERK]{{\footnotesize \input{gauss_cpu_wrt_errphase}}}
\goodgap  \subfigure[Gauss-Lawson]{{\footnotesize\input{lawson_cpu_wrt_errphase}}}\\
\subfigure[Splitting]{{\footnotesize\input{split_cpu_wrt_errphase}}}
  \caption{Evolution of the CPU time with respect to $\mathcal{E}_{P,h}$ error {\color{black}for problem \eqref{eq:schro_num}} {\color{black}(see Figure \ref{fig:legends} for legends)}}
  \label{fig:CPU}
\end{figure}
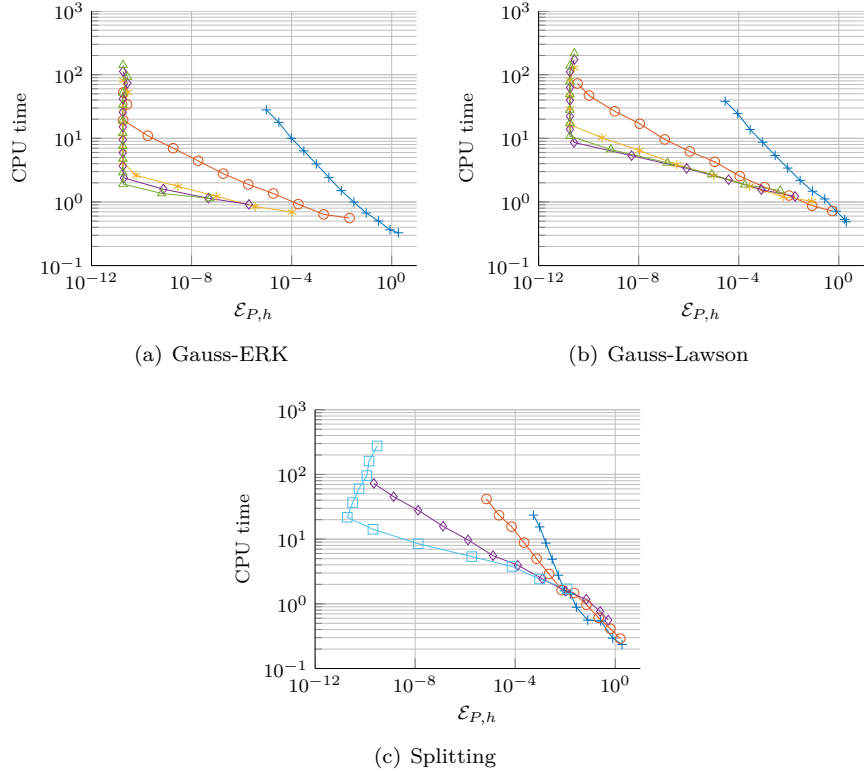
 
{\color{black} Another interesting question is the behaviour of conserved quantities (mass, energy) by the methods over long times. For this numerical study, the final time is set to $5000$ and to avoid any hazardous behaviours due to boundary conditions, we choose as initial condition $\psi(0,x)=|\sin(x)|$ and $x\in [-\pi,\pi)$. With this choice, the datum belongs to $H^1(\T_{2\pi}^{1})$.
The number of Fourier modes is $2^{10}$ and we take $h=10^{-2}$. We are interested in the quantities $\mathcal{E}_{M,h}(t_n)$ defined in \eqref{eq:errmass} and $\mathcal{E}_{E,h}(t_n)$ defined in \eqref{eq:errener}.
We plot in Figure \ref{fig:LongErrMass} the evolution of $\mathcal{E}_{M,h}(t_n)$ for both Gauss-ERK and Gauss-Lawson methods. The $L^2$-norm conservation holds up to roundoff for Gauss-Lawson methods as predicted by Theorem \ref{quadinv}. Almost preservation of the $L^2$-norm over long times seems to hold for Gauss-ERK methods. In contrast, the evolution of $\mathcal{E}_{E,h}(t_n)$ (Figure \ref{fig:LongErrMass}) 
shows that the Gauss-Lawson scheme breaks down in such a situation. The Gauss-ERK scheme seems to almost preserve the energy and is clearly superior, with respect to energy preservation, to the Lawson scheme for long time simulations. The break down phenomenon of the Lawson scheme can be reduced by selecting a smaller time step ($h=10^{-3}$, see Figure \ref{fig:LongErrEner2}). 
However, the numerical cost is dramatically more important compared to a Gauss-ERK scheme.
}
\begin{figure}[!h]
  \centering
  \subfigure[Gauss-ERK]{\includegraphics[width=.48\textwidth]{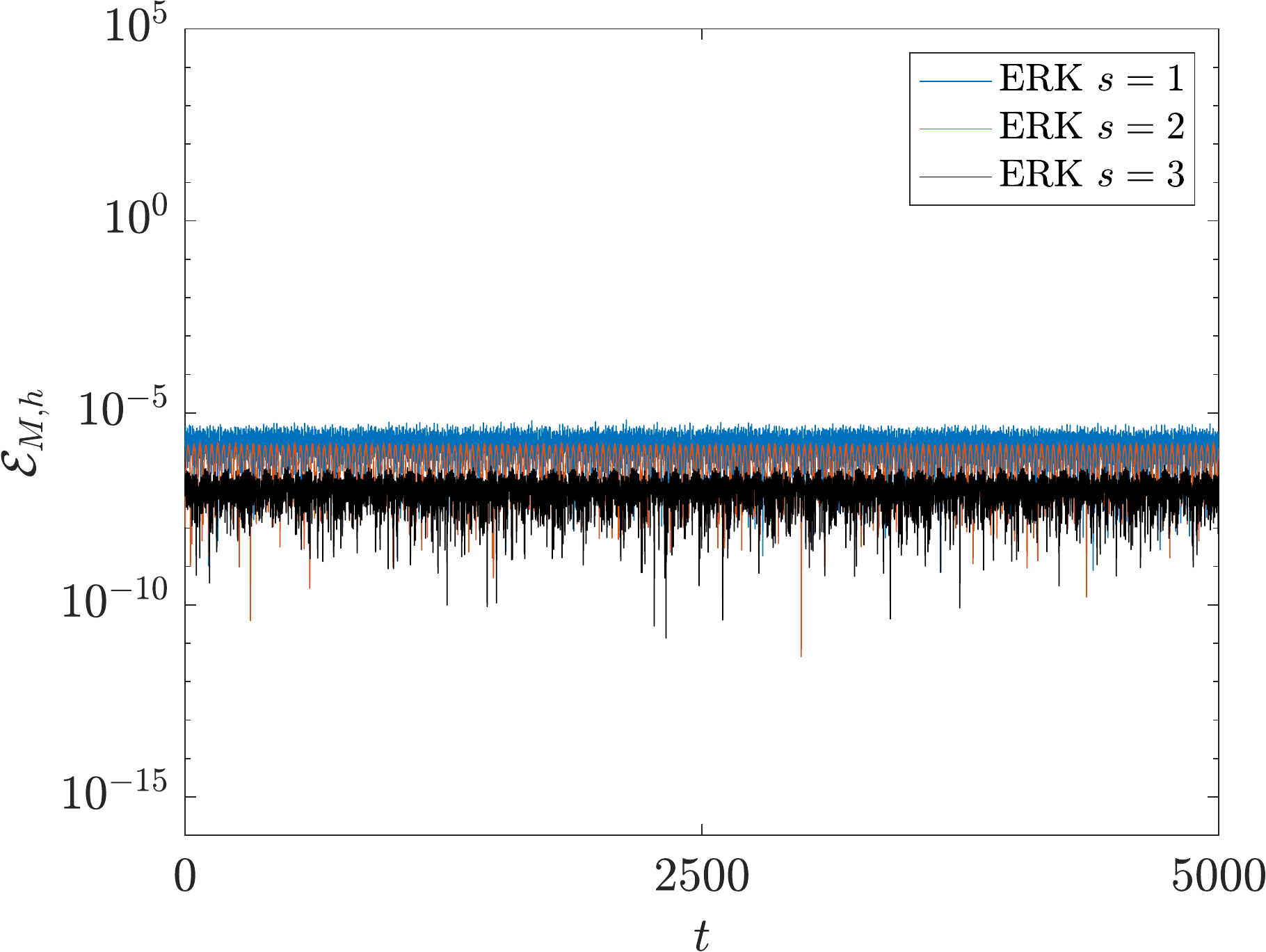}}
\goodgap  \subfigure[Gauss-Lawson]{\includegraphics[width=.48\textwidth]{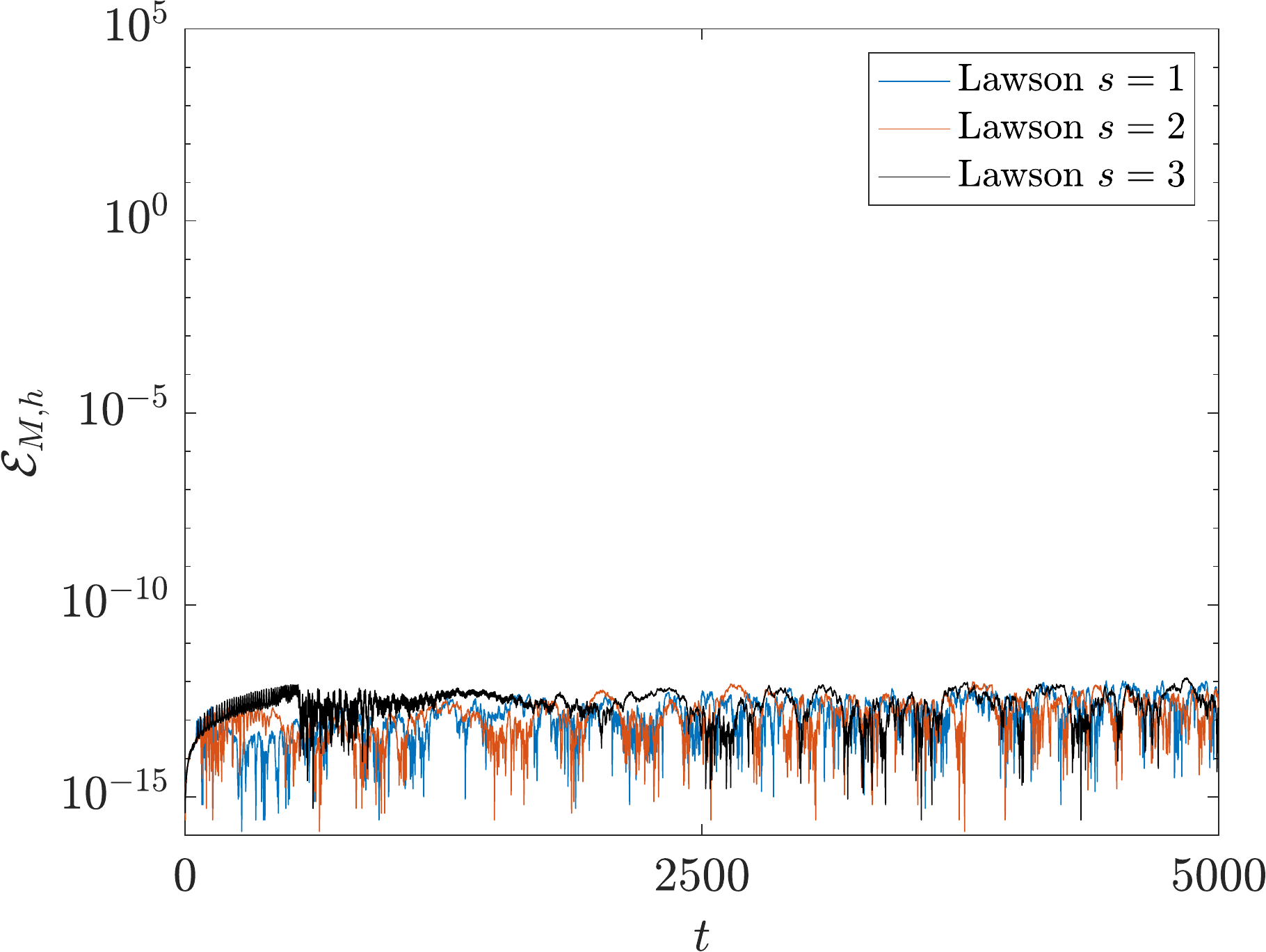}}\\
  \subfigure[Gauss-ERK]{\includegraphics[width=.48\textwidth]{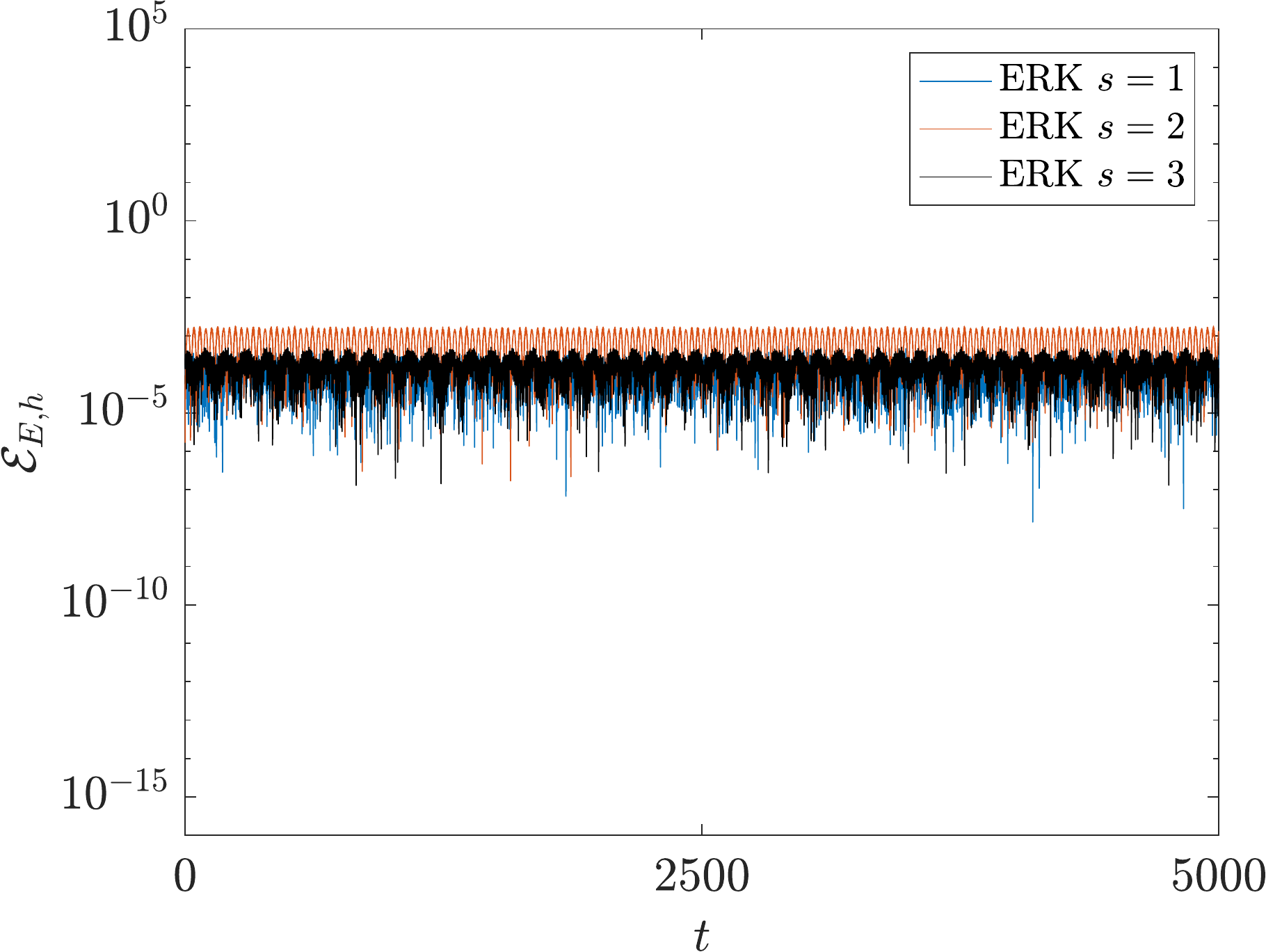}}
\goodgap  \subfigure[Gauss-Lawson]{\includegraphics[width=.48\textwidth]{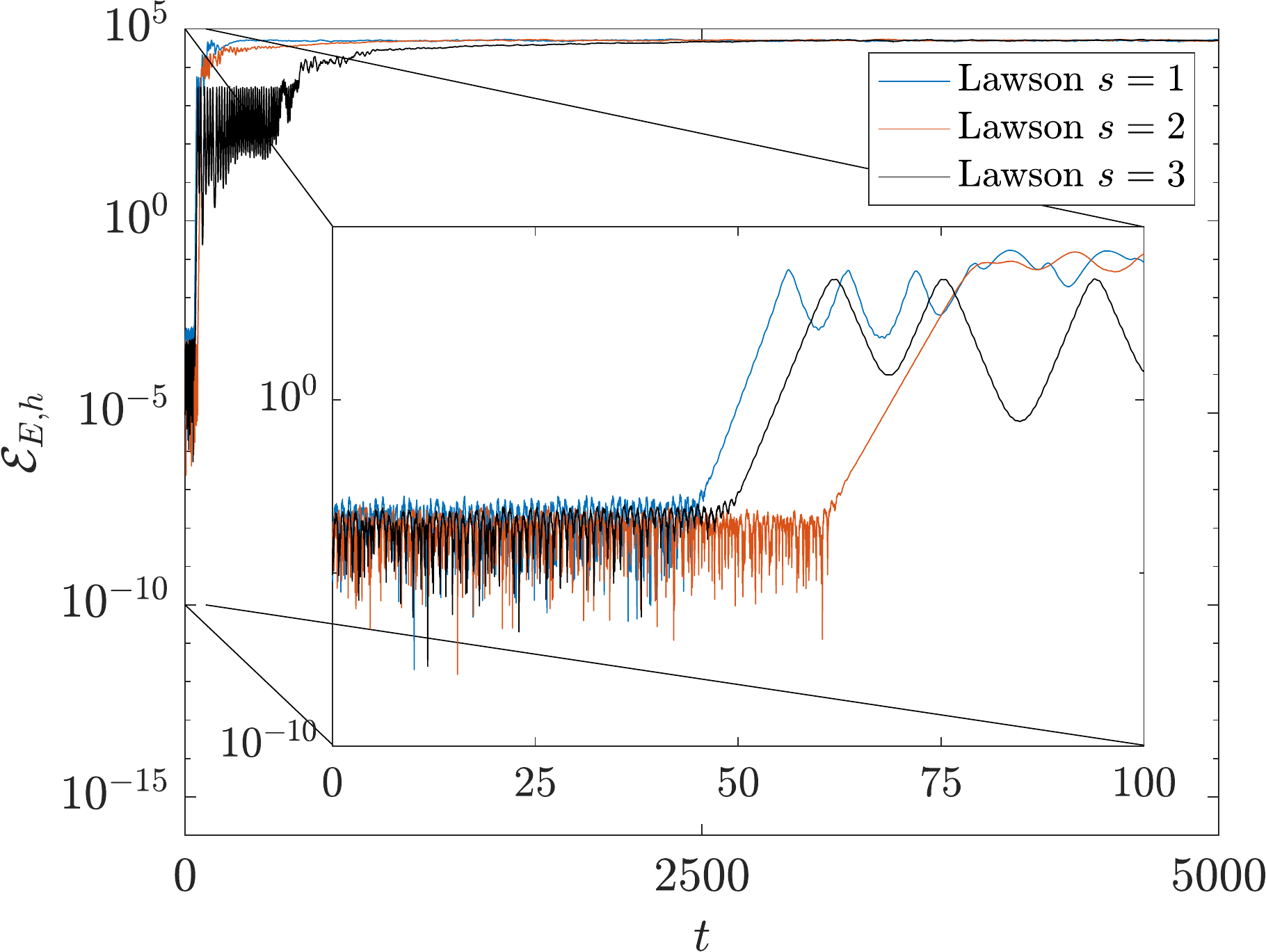}}
  \caption{{\color{black}Evolution of $\mathcal{E}_{M,h}(t_n)$ (figures (a) and (b)) and $\mathcal{E}_{E,h}(t_n)$ (figures (c) and (d)), $h=10^{-2}$}}
  \label{fig:LongErrMass}
\end{figure}


\begin{figure}[!h]
  \centering
  \includegraphics[width=.48\textwidth]{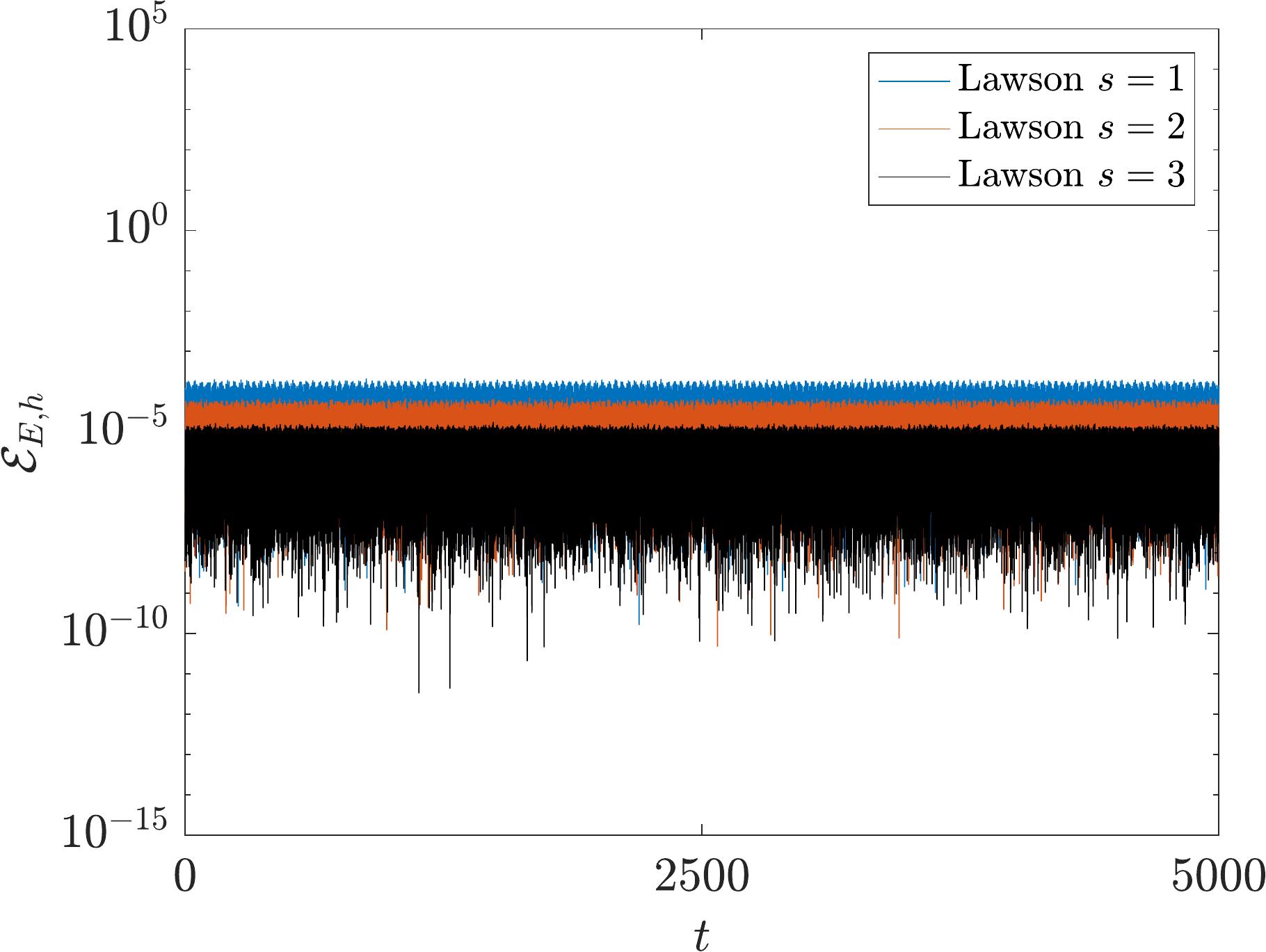}
  \caption{Evolution of $\mathcal{E}_{E,h}(t_n)$, $h=10^{-3}$, Gauss-Lawson}
  \label{fig:LongErrEner2}
\end{figure}

{\color{black}In conclusion, even if the Gauss-ERK schemes do not preserve theoretically neither the $L^2$ norm nor the energy, they allow us to preserve them numerically for reasonably small time step $h$ even in long time simulations and clearly are the best methods for one dimensional simulations of the cubic NLS equations.}

\subsection{One dimensional example: \textcolor{black}{a NLS equation with time-dependent potential and cubic-plus-quintic nonlinearity}}
\label{Comp1d2}

{\color{black}We consider here the following one-dimensional NLSE (see e.g. \cite{BeCu})
\begin{equation}
  \label{eq:BeCu}
  i \partial_t \psi = -\partial_x^2\psi + V\psi + G_1|\psi|^2\psi + G_2|\psi|^4\psi,
\end{equation}
with a cubic-plus-quintic nonlinearity and an external time-dependent potential given by
$V(t,x)=\frac x2 \omega^2\cos(\omega t)$.

The splitting schemes have to be applied carefully for this equation since the step associated to the potential term leads to a non-autonomous problem. However, the Gauss-ERK and Gauss-Lawson methods can be applied straightforwardly for this test case.

We know in this case an exact solution (the bright soliton) given by 
$$
\displaystyle \psi_{\textrm{ex}}(t,x)=\eta
\frac{ \exp\Big(i\Big( -\frac\omega2 x \sin(\omega t + \beta_0) - \frac{\omega^2}8 t
+  \frac{\omega}{16}\sin(2\omega t + 2\beta_0) - E_{\textrm{c}}t \Big)\Big)}
{\left( \sqrt{1-b}\cosh(2\sqrt{-E_{\textrm{c}}}(x-\cos(\omega t))) + 1 \right)^{1/2}}.
$$
The final time $T$ is chosen to be $5$. The parameters of the reference solution are fixed to
$$
G_1=-2, \ G_2=\frac12, \ \omega=2, \ E_{\textrm{c}}=-1, \ \beta_0=0,
\ \eta = \sqrt{\frac{4 E_{\textrm{c}}}{G_1}}, \ b = - 16 \frac{E_{\textrm{c}} G_2}{3 G_1^2}.
$$
The spatial computational domain is $\mathcal{D}:=]-32,32[$,
with $N_x=2^{11}$ Fourier modes.
For this example, the energy is not conserved.
We therefore present 
in Figure \ref{fig:BeCuEPh}
only the evolution of $\mathcal{E}_{P,h}$
and $\mathcal{E}_{M,h}$ errors.
{\color{black}
We only draw the error curves for $s\leq 4$
for Gauss-ERK and Gauss-Lawson methods since the curves for $s=5$
are identical to the ones for $s=4$.
and since Gauss-Lawson methods preserve the mass up to round-off
(see Theorem \ref{quadinv}).
}
Like for the cubic NLS equation in subsection \ref{Comp1d},
we have good decay rate of the $\mathcal{E}_{P,h}$ curves.
}

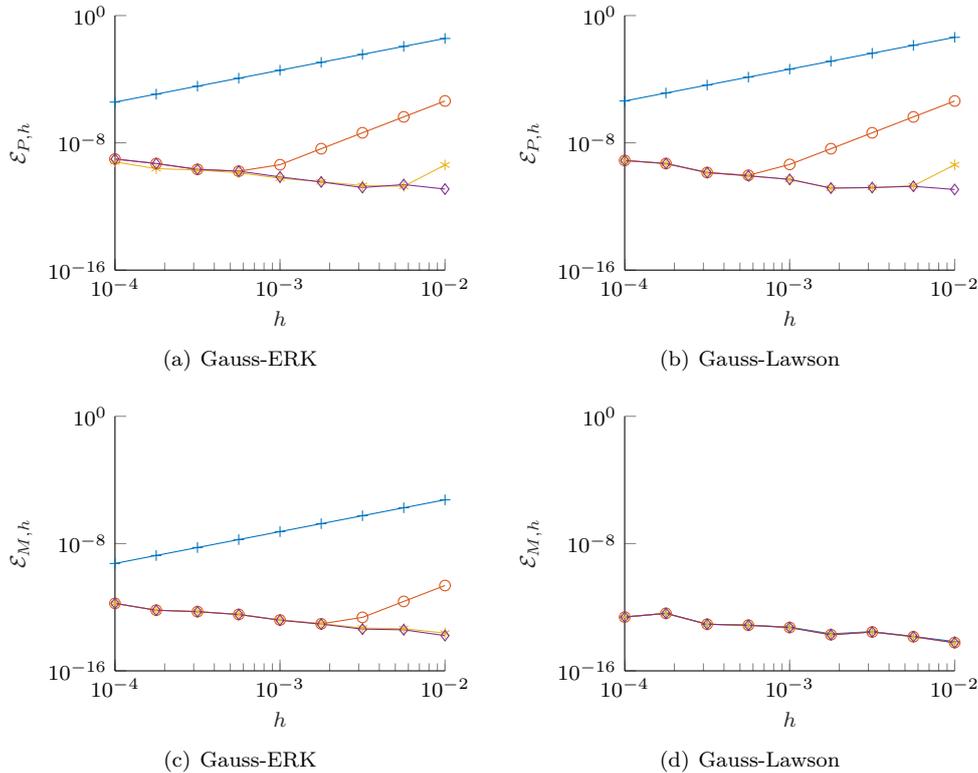
\begin{figure}[!h]
  \centering
  \subfigure[Gauss-ERK]{{\footnotesize\input{err_phase_ERK_cas_test_pot_var}}}
 \goodgap  \subfigure[Gauss-Lawson]{{\footnotesize\input{err_phase_Lawson_cas_test_pot_var}}}\\
  \subfigure[Gauss-ERK]{{\footnotesize\input{err_mass_ERK_cas_test_pot_var}}}
 \goodgap  \subfigure[Gauss-Lawson]{{\footnotesize\input{err_mass_Lawson_cas_test_pot_var}}}

  \caption{{\color{black}Evolution of $\mathcal{E}_{P,h}$ (figures (a) and (b)) and $\mathcal{E}_{M,h}$ (figures (c) and (d)) errors with respect to the time step for problem \eqref{eq:BeCu} (see Figure \ref{fig:legends} for legends)}}
  \label{fig:BeCuEPh}
\end{figure}

\subsection{Two dimensional simulations}
\label{Comp2d}

The first 2D example concerns the equation
\begin{equation}
  \label{eq:NLS2D}
  \partial_t \psi= i \Delta \psi + i |\psi|^2 \psi, \quad (t,\mathbf{x}) \in [0,T]\times \R^2,
\end{equation}
which is a nonlinear Schr\"odinger equation without confinement or rotation.
We look for a solution of
the form
$$
\psi(t,\mathbf{x})=e^{it} \Theta(\mathbf{x}),
$$
where $\Theta$ solves
$$
\Delta \Theta +  |\Theta|^2 \Theta = \Theta,
$$
with $\lim_{\|\mathbf{x}\|\to \infty} \Theta(\mathbf{x})=0$. Since we do not have access to an exact solution $\Theta$ to this problem, we generate it through a classical shooting point method \cite{dimenza_2009}. Since the decay of this kind of solutions is slow, we consider $(x,y) \in [-38,38]^2$ with $P=9$.

As for one dimensional simulations, we evaluate several errors for the different methods. We generalize the error functions $\mathcal{E}_{P,h}$, $\mathcal{E}_{M,h}$ and $\mathcal{E}_{E,h}$ defined in \eqref{eq:errphase}, \eqref{eq:errmass} and \eqref{eq:errener} to two dimensional cases. We only present here simulations for ERK and Lawson methods with $s=1$, $2$ and $3$ since we have seen on one dimensional experiments that we do not gain much more precision for $s>3$. \textcolor{black}{We also restrict ourselves to the presentation of the $\mathcal{E}_{P,h}$ error
displayed in Figure \ref{fig:EPh_2D} since the $\mathcal{E}_{M,h}$ and $\mathcal{E}_{E,h}$ errors lead to similar results to the ones presented in Figures \ref{fig:EMh} and \ref{fig:EEh}.}
\begin{figure}[!h]
  \centering
  \subfigure[Gauss-ERK]{{\footnotesize \input{gauss_errphase_wrt_dt_2D}}\label{fig:EPhERK_2D}}
\goodgap  \subfigure[Gauss-Lawson]{{\footnotesize\input{lawson_errphase_wrt_dt_2D}}\label{fig:EPhLaw_2D}}\\
\subfigure[Splitting]{{\footnotesize\input{split_errphase_wrt_dt_2D}}\label{fig:EPhSplit_2D}}
  \caption{Evolution of $\mathcal{E}_{P,h}$ error with respect to the time step {\color{black}for problem \eqref{eq:NLS2D}} {\color{black}(see Figure \ref{fig:legends}) for legends}}
  \label{fig:EPh_2D}
\end{figure}
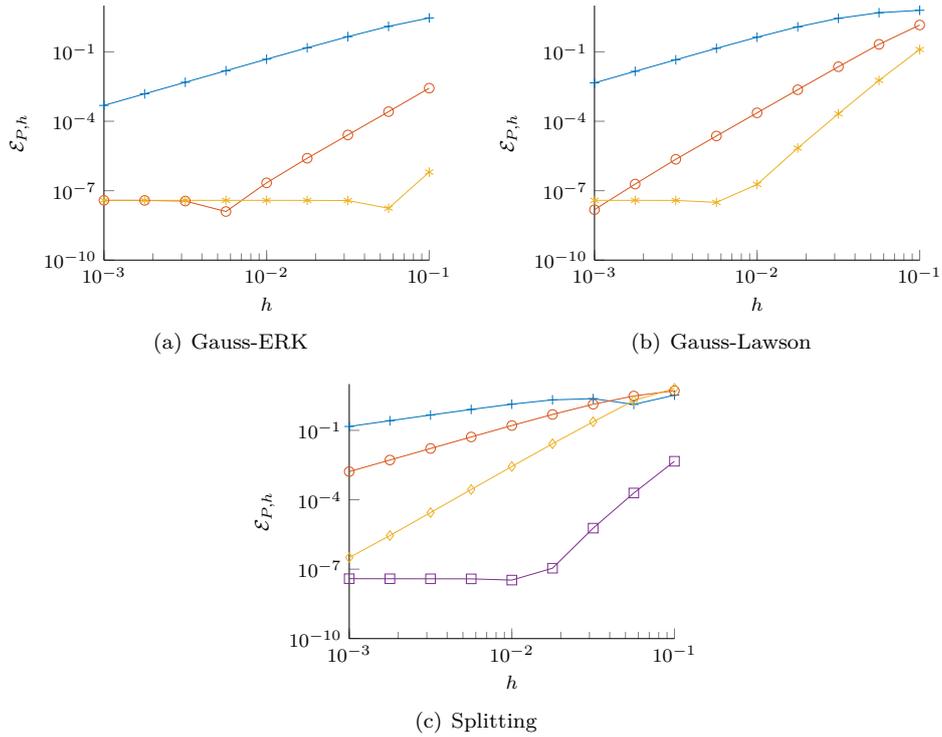

\textcolor{black}{ We also propose in this subsection to illustrate the efficiency and versatility of the methods
on the simulation of a soliton for the two-dimensional Schr\"odinger equation
and the simulation of the evolution of a rotating Bose--Einstein condensate.
Thereby, }our second 2D numerical example corresponds to a rotating BEC modelled
by equation \eqref{GPEeq}.
We reproduce here with the Gauss-ERK method of order 6 the simulations
realized in \cite{BMTZ2013} for \eqref{GPEeq}, see Figure \ref{fig:bec}.
The parameters are $\beta=1000$, $\Omega=0.9$, and the potential is
\eqref{potharmonic} with $\gamma_x=1.05$ and $\gamma_y=0.95$.
The computational domain is $(-16,16)^2$ with $2^9$ Fourier modes in each
direction.
The time step is $h=10^{-3}$ with a final time $T=7$.
The initial datum is the ground state of the stationary equation and
was generated using Matlab toolbox
GPELab\footnote{\url{http://gpelab.math.cnrs.fr}}.
We recover the same behaviour and we get good conservation properties.

\begin{figure}[!h]
  \centering
  \begin{tabular}{ccc}
    \includegraphics[width=.28\textwidth]{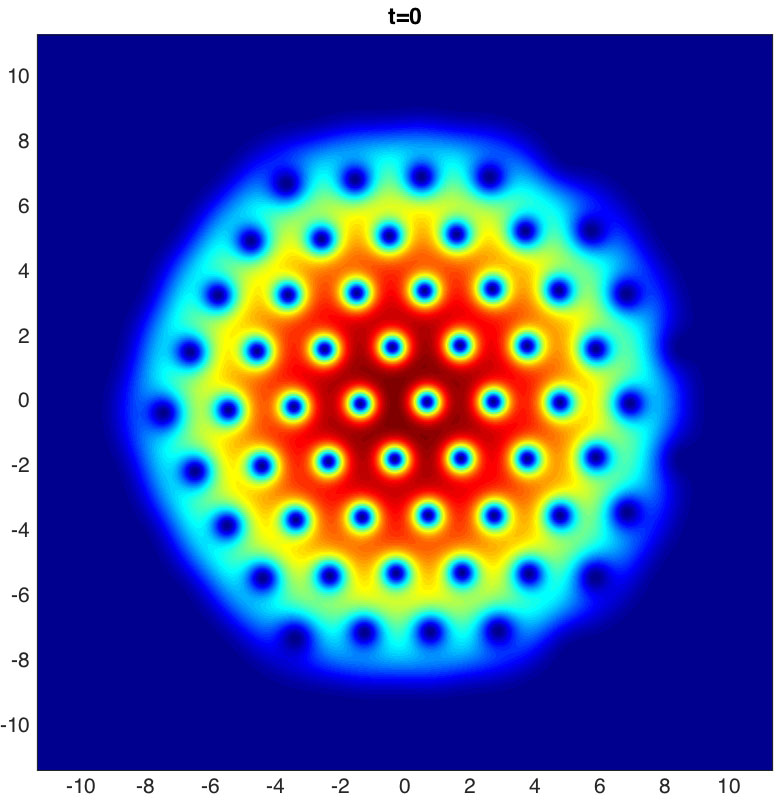} & \includegraphics[width=.28\textwidth]{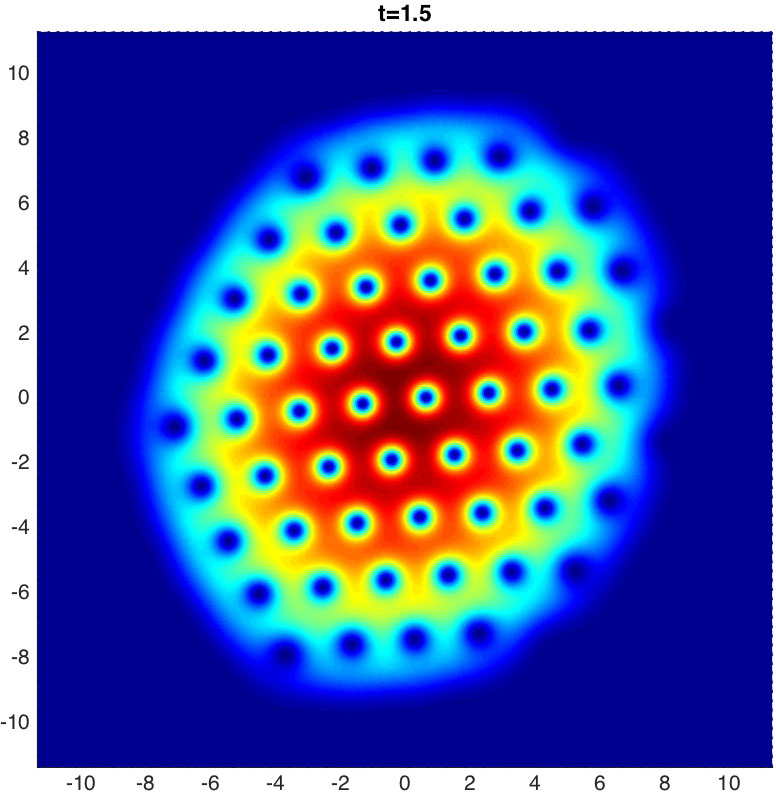} & \includegraphics[width=.28\textwidth]{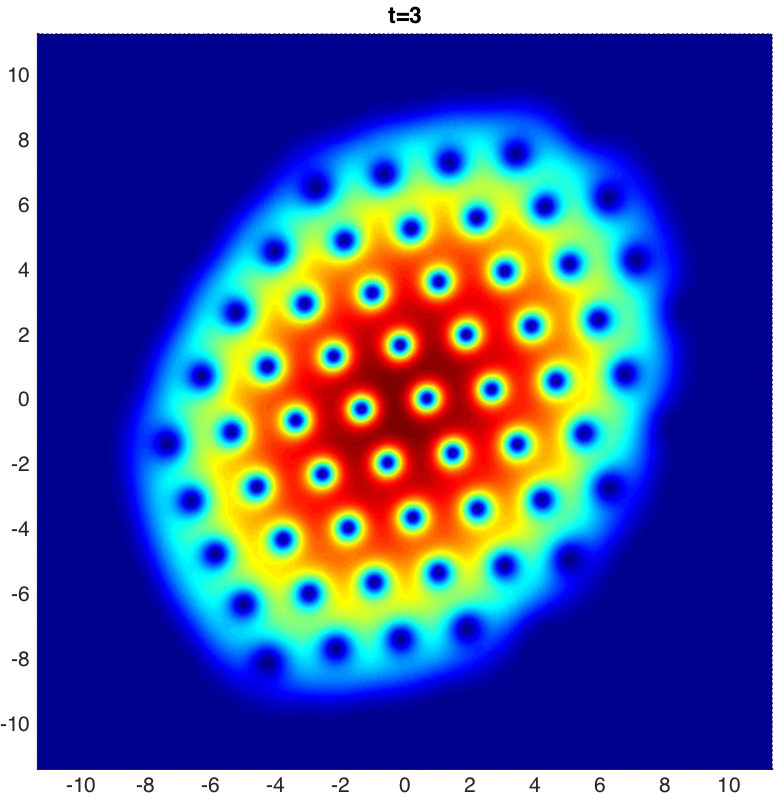} \\
    \includegraphics[width=.28\textwidth]{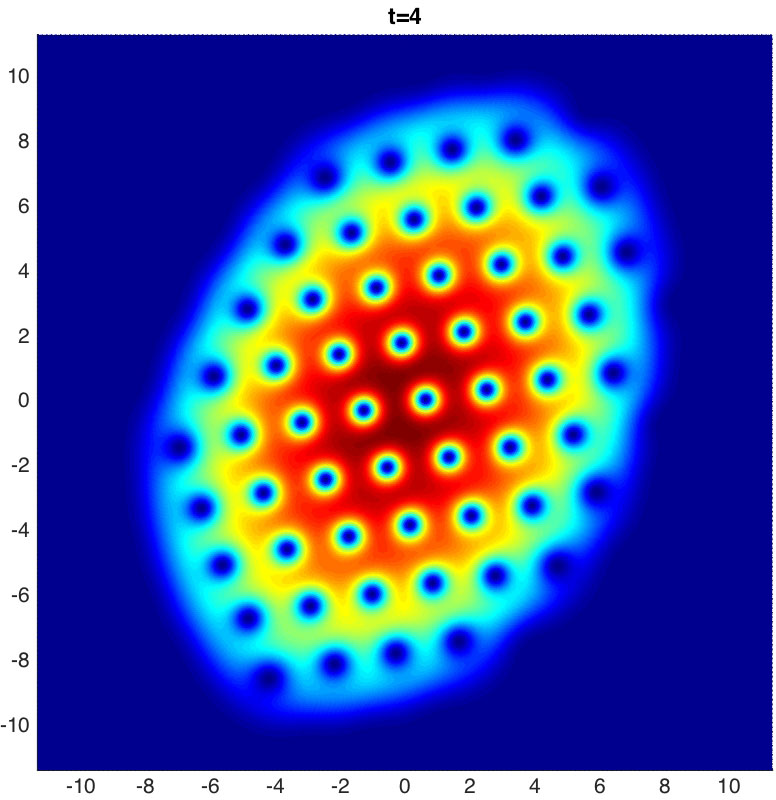} & \includegraphics[width=.28\textwidth]{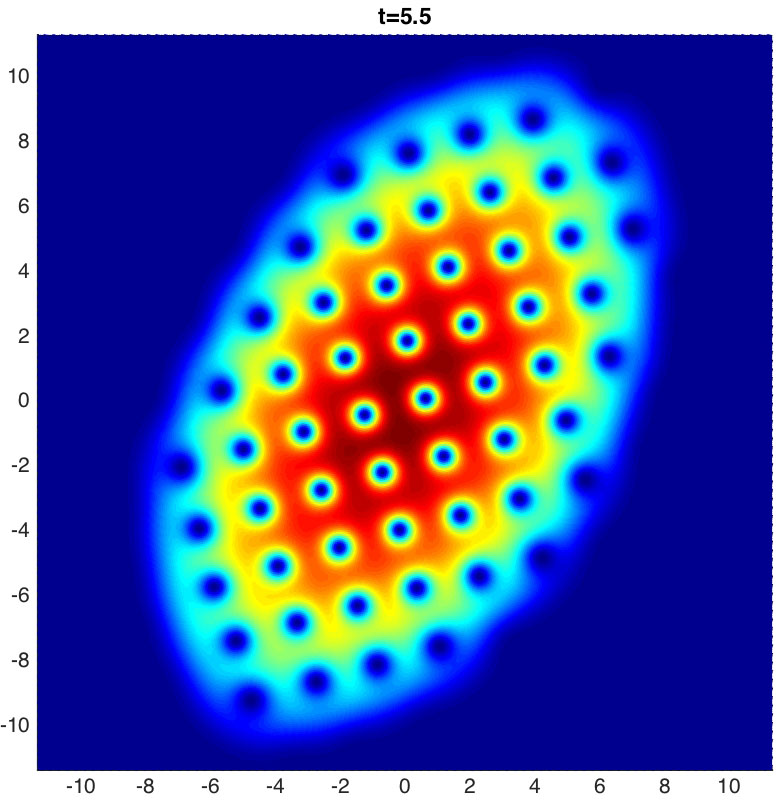} & \includegraphics[width=.28\textwidth]{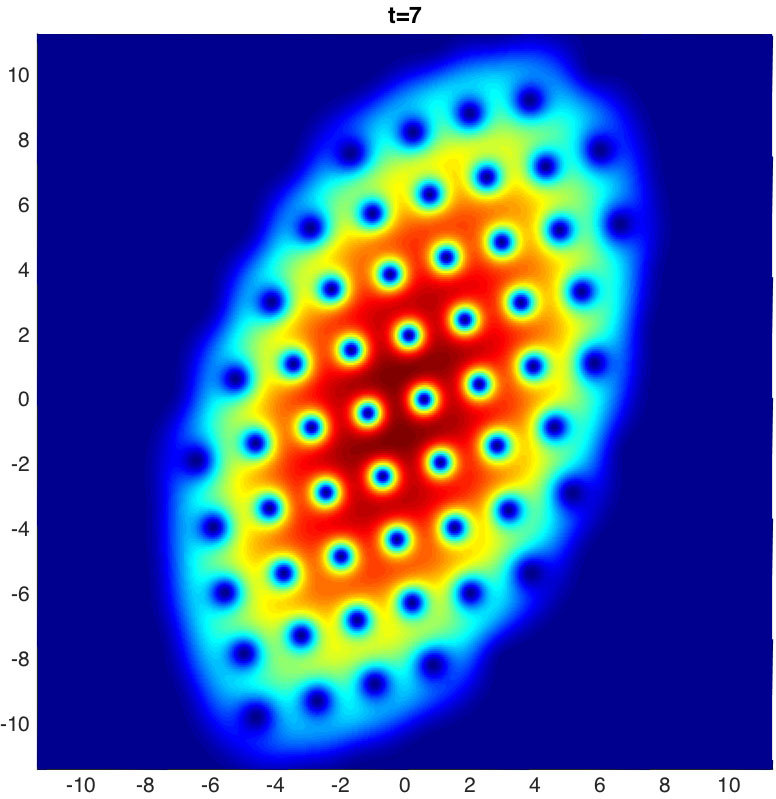} \\
  \end{tabular}
  \caption{Contour plots of the density function $|\varphi(t,\mathbf{x})|^2$ in a rotating BEC.}
  \label{fig:bec}
\end{figure}

\section*{Conclusion}

We presented ERK and Lawson methods that allow {\color{black}us} to compute numerical
solutions of NLS equations modelling a rotating BEC in a fairly general
setting.
This procedure allows {\color{black}us} to derive neatly high order methods
(in contrast to finding coefficients for high order splitting methods with
small error constants, for example).
We proved finite-time convergence in Sobolev norms for these methods
in a simplified framework.
We compared the numerical results provided by these methods to that obtained
via classical splitting methods in several configurations: 1D problems,
2D problems without confinement or rotation, 2D realistic problems with
rotation and confinement.
It is {\color{black}noteworthy} that all the methods presented in this paper
allow to deal with non-autonomous problems (no matter whether the
lack of autonomy comes from the physical situation or a change of unknown).

When it comes to finite time accuracy, Gauss-ERK methods outperformed
Lawson methods and splitting methods in all cases, since they have
very low error constants numerically.
Moreover, even if they do not preserve mass exactly (up to roundoff errors),
their relative error on the mass is comparable to that of Gauss-Lawson
and splitting methods (which preserve the mass up to roundoff errors) for
reasonably small time steps (see Figure \ref{fig:EMh} for time steps of order
$10^{-4}$). This can be explained by the accumulation of roundoff errors when
$h$ gets smaller for the methods preserving mass exactly.

When it comes to computational times, Gauss-ERK methods, though implicit,
also outperformed Lawson and splitting methods (see Figure \ref{fig:CPU})
on the $1D$ example. To be completely fair, one must say that
ERK methods {(and, to a lesser extent, Lawson methods)} require the precomputation of some coefficients.
This computation only needs to be carried out once, before 
one starts the time stepping procedure, and can be parallelized
if necessary.

With respect to average or long time behaviour, in contrast to splitting
methods, ERK and Lawson methods do not show resonances phenomena \cite{DuFa2007,FaJe2015}.

\section*{Acknowledgements}
This work was partially supported by the French ANR grant ANR-12-MONU-0007-02 BECASIM (''Mod\`eles Num\'eriques'' call), by the Inria teams MEPHYSTO and RAPSODI and by the Labex CEMPI  (ANR-11-LABX-0007-01). 
We would like to thank Romain Duboscq for his valuable help to build the initial datum for the second two dimensional test case.
{\color{black} The authors also thank the referees for their valuable
suggestions and comments.}


\input{biblio}

\end{document}

%% file: legend.tex
%
%
\definecolor{mycolor1}{rgb}{0.00000,0.44700,0.74100}%
\definecolor{mycolor2}{rgb}{0.85000,0.32500,0.09800}%
\definecolor{mycolor3}{rgb}{0.92900,0.69400,0.12500}%
\definecolor{mycolor4}{rgb}{0.49400,0.18400,0.55600}%
\definecolor{mycolor5}{rgb}{0.46600,0.67400,0.18800}%
\begin{tikzpicture}

  \draw (0,0) rectangle (1.9,2.6);
  \draw[color=mycolor5,solid,mark=o] (0.1,0.3) -- (0.7,0.3);
  \draw[color=mycolor5] plot[mark=triangle,mark size=2pt] (0.4,0.3);
  \draw (0.8,0.3) node[right] {$s = 5$};

  \draw[color=mycolor4,solid,mark=o] (0.1,0.8) -- (0.7,0.8);
  \draw[color=mycolor4] plot[mark=diamond,mark size=2pt] (0.4,0.8);
  \draw (0.8,0.8) node[right] {$s = 4$};

  \draw[color=mycolor3,solid,mark=o] (0.1,1.3) -- (0.7,1.3);
  \draw[color=mycolor3] plot[mark=asterisk,mark size=2pt] (0.4,1.3);
  \draw (0.8,1.3) node[right] {$s = 3$};

  \draw[color=mycolor2,solid,mark=o] (0.1,1.8) -- (0.7,1.8);
  \draw[color=mycolor2] plot[mark=o,mark size=2pt] (0.4,1.8);
  \draw (0.8,1.8) node[right] {$s = 2$};

  \draw[color=mycolor1,solid,mark=+] (0.1,2.3) -- (0.7,2.3);
  \draw[color=mycolor1] plot[mark=+,mark size=2pt] (0.4,2.3);
  \draw (0.8,2.3) node[right] {$s = 1$};

\end{tikzpicture}%

%% file: legend_split.tex
%
%
\definecolor{mycolor1}{rgb}{0.00000,0.44700,0.74100}%
\definecolor{mycolor2}{rgb}{0.85000,0.32500,0.09800}%
\definecolor{mycolor3}{rgb}{0.92900,0.69400,0.12500}%
\definecolor{mycolor4}{rgb}{0.49400,0.18400,0.55600}%
\definecolor{mycolor5}{rgb}{0.46600,0.67400,0.18800}%
\definecolor{mycolor6}{rgb}{0.30100,0.74500,0.93300}%
\begin{tikzpicture}

  \draw (0,1.04) rectangle (3.3,3.2);

  \draw[color=mycolor6,solid,mark=square] (0.1,1.3) -- (0.7,1.3);
  \draw[color=mycolor6] plot[mark=square,mark size=2pt] (0.4,1.3);
  \draw (0.8,1.3) node[right] {Splitting O$(6)$};

  \draw[color=mycolor4,solid,mark=diamond] (0.1,1.8) -- (0.7,1.8);
  \draw[color=mycolor4] plot[mark=diamond,mark size=2pt] (0.4,1.8);
  \draw (0.8,1.8) node[right] {Splitting O$(4)$};

  \draw[color=mycolor2,solid,mark=o] (0.1,2.3) -- (0.7,2.3);
  \draw[color=mycolor2] plot[mark=o,mark size=2pt] (0.4,2.3);
  \draw (0.8,2.3) node[right] {Splitting O$(2)$};

  \draw[color=mycolor1,solid,mark=+] (0.1,2.8) -- (0.7,2.8);
  \draw[color=mycolor1] plot[mark=+,mark size=2pt] (0.4,2.8);
  \draw (0.8,2.8) node[right] {Splitting O$(1)$};

\end{tikzpicture}%

%% file: gauss_errphase_wrt_dt.tex
%
%
\definecolor{mycolor1}{rgb}{0.00000,0.44700,0.74100}%
\definecolor{mycolor2}{rgb}{0.85000,0.32500,0.09800}%
\definecolor{mycolor3}{rgb}{0.92900,0.69400,0.12500}%
\definecolor{mycolor4}{rgb}{0.49400,0.18400,0.55600}%
\definecolor{mycolor5}{rgb}{0.46600,0.67400,0.18800}%
\begin{tikzpicture}[scale=0.9]

\begin{axis}[%
width=4.754601cm,
height=3.75cm,
at={(0cm,0cm)},
scale only axis,
xmode=log,
xmin=0.0001,
xmax=0.1,
xminorticks=true,
xlabel={$h$},
ymode=log,
ymin=1e-12,
ymax=10,
yminorticks=true,
ylabel={$\mathcal{E}_{P,h}$},
axis x line*=bottom,
axis y line*=left
] ()
\addplot [color=mycolor1,solid,mark=+,mark options={solid},forget plot]
  table[row sep=crcr]{%
0.1	1.84557116883949\\
0.0562341325190349	0.854659839575447\\
0.0316227766016838	0.296255732876074\\
0.0177827941003892	0.0961408191943907\\
0.01	0.0306557640598316\\
0.00562341325190349	0.00971569216073048\\
0.00316227766016838	0.00307488382984994\\
0.00177827941003892	0.000972785095467171\\
0.001	0.000307611457316651\\
0.000562341325190349	9.72730893833666e-05\\
0.000316227766016838	3.07612955269807e-05\\
0.000177827941003892	9.72781564205045e-06\\
};
\addplot [color=mycolor2,solid,mark=o,mark options={solid},forget plot]
  table[row sep=crcr]{%
0.1	0.0206578562543729\\
0.0562341325190349	0.0018965971253223\\
0.0316227766016838	0.000184889727845496\\
0.0177827941003892	1.84075541001901e-05\\
0.01	1.83966489515784e-06\\
0.00562341325190349	1.83875690309045e-07\\
0.00316227766016838	1.83838390994948e-08\\
0.00177827941003892	1.83455490448779e-09\\
0.001	1.78013931350978e-10\\
0.000562341325190349	1.84182054965474e-11\\
0.000316227766016838	2.62749707880861e-11\\
0.000177827941003892	1.87358119529605e-11\\
};
\addplot [color=mycolor3,solid,mark=asterisk,mark options={solid},forget plot]
  table[row sep=crcr]{%
0.1	0.000104113083490949\\
0.0562341325190349	3.46924799492737e-06\\
0.0316227766016838	9.75706314906392e-08\\
0.0177827941003892	2.81564462517891e-09\\
0.01	6.03476322824619e-11\\
0.00562341325190349	1.78668283334528e-11\\
0.00316227766016838	1.78793120866082e-11\\
0.00177827941003892	1.78817098234705e-11\\
0.001	1.78802214977198e-11\\
0.000562341325190349	1.78779835636396e-11\\
0.000316227766016838	2.80640001297142e-11\\
0.000177827941003892	1.87126209104355e-11\\
};
\addplot [color=mycolor4,solid,mark=diamond,mark options={solid},forget plot]
  table[row sep=crcr]{%
0.1	1.95975025292956e-06\\
0.0562341325190349	4.69043548935922e-08\\
0.0316227766016838	7.47730425142384e-10\\
0.0177827941003892	1.97998503867999e-11\\
0.01	1.7866008510553e-11\\
0.00562341325190349	1.78657825419057e-11\\
0.00316227766016838	1.78793582225984e-11\\
0.00177827941003892	1.78820304592522e-11\\
0.001	1.78806149103514e-11\\
0.000562341325190349	1.787798555604e-11\\
0.000316227766016838	2.72759731563889e-11\\
0.000177827941003892	1.82576059116287e-11\\
};
\addplot [color=mycolor5,solid,mark=triangle,mark options={solid},forget plot]
  table[row sep=crcr]{%
0.1	5.37233329755416e-08\\
0.0562341325190349	6.39038608078418e-10\\
0.0316227766016838	1.81089773077742e-11\\
0.0177827941003892	1.78653608612102e-11\\
0.01	1.7866527146341e-11\\
0.00562341325190349	1.78656282966088e-11\\
0.00316227766016838	1.78790630111594e-11\\
0.00177827941003892	1.78820063289106e-11\\
0.001	1.78803180285659e-11\\
0.000562341325190349	1.78779312278978e-11\\
0.000316227766016838	2.72448712009838e-11\\
0.000177827941003892	1.8150857704615e-11\\
};
\end{axis}
\end{tikzpicture}%

%% file: lawson_errphase_wrt_dt.tex
%
%
\definecolor{mycolor1}{rgb}{0.00000,0.44700,0.74100}%
\definecolor{mycolor2}{rgb}{0.85000,0.32500,0.09800}%
\definecolor{mycolor3}{rgb}{0.92900,0.69400,0.12500}%
\definecolor{mycolor4}{rgb}{0.49400,0.18400,0.55600}%
\definecolor{mycolor5}{rgb}{0.46600,0.67400,0.18800}%
\begin{tikzpicture}[scale=0.9]

\begin{axis}[%
width=4.754601cm,
height=3.75cm,
at={(0cm,0cm)},
scale only axis,
xmode=log,
xmin=0.0001,
xmax=0.1,
xminorticks=true,
xlabel={$h$},
ymode=log,
ymin=1e-12,
ymax=100,
yminorticks=true,
ylabel={$\mathcal{E}_{P,h}$},
axis x line*=bottom,
axis y line*=left,
legend style={at={(0.03,0.97)},anchor=north west,legend cell align=left,align=left,draw=white!15!black}
]
\addplot [color=mycolor1,solid,mark=+,mark options={solid}]
  table[row sep=crcr]{%
0.1	1.99697062079537\\
0.0562341325190349	1.72977258988882\\
0.0316227766016838	0.781322874394886\\
0.0177827941003892	0.274041501642724\\
0.01	0.0894217269415249\\
0.00562341325190349	0.0285418681884722\\
0.00316227766016838	0.00905316108361094\\
0.00177827941003892	0.00286610508064599\\
0.001	0.000906511662868099\\
0.000562341325190349	0.000286677662672945\\
0.000316227766016838	9.06599850424159e-05\\
0.000177827941003892	2.8670152526422e-05\\
};

\addplot [color=mycolor2,solid,mark=o,mark options={solid}]
  table[row sep=crcr]{%
0.1	0.546383116130535\\
0.0562341325190349	0.086429585015011\\
0.0316227766016838	0.0102215160988692\\
0.0177827941003892	0.00107665863494139\\
0.01	0.000109607125261287\\
0.00562341325190349	1.10185325712024e-05\\
0.00316227766016838	1.10380613542212e-06\\
0.00177827941003892	1.1045986975924e-07\\
0.001	1.10465855504683e-08\\
0.000562341325190349	1.1084758121233e-09\\
0.000316227766016838	1.03415984199217e-10\\
0.000177827941003892	3.59119986430842e-11\\
};

\addplot [color=mycolor3,solid,mark=asterisk,mark options={solid}]
  table[row sep=crcr]{%
0.1	0.084080306372219\\
0.0562341325190349	0.00594354158902568\\
0.0316227766016838	0.000275076621490442\\
0.0177827941003892	1.00092814406047e-05\\
0.01	3.34821472290322e-07\\
0.00562341325190349	1.07891455607122e-08\\
0.00316227766016838	3.41385128151359e-10\\
0.00177827941003892	1.79169987680997e-11\\
0.001	1.78809705832022e-11\\
0.000562341325190349	1.78749034598323e-11\\
0.000316227766016838	1.78750346931266e-11\\
0.000177827941003892	2.69130211965167e-11\\
};

\addplot [color=mycolor4,solid,mark=diamond,mark options={solid}, forget plot]
  table[row sep=crcr]{%
0.1	0.0178006745529981\\
0.0562341325190349	0.000798024145417313\\
0.0316227766016838	3.93617425486699e-05\\
0.0177827941003892	8.31124142997175e-07\\
0.01	5.21351394121515e-09\\
0.00562341325190349	2.59684433879161e-11\\
0.00316227766016838	1.78801690884268e-11\\
0.00177827941003892	1.78832282264235e-11\\
0.001	1.78808683026381e-11\\
0.000562341325190349	1.78749088206945e-11\\
0.000316227766016838	1.78750230945903e-11\\
0.000177827941003892	2.69388191015191e-11\\
};

\addplot [color=mycolor5,solid,mark=triangle,mark options={solid}]
  table[row sep=crcr]{%
0.1	0.00456131470895431\\
0.0562341325190349	0.000174326339488176\\
0.0316227766016838	8.3667778358743e-06\\
0.0177827941003892	1.43326425642903e-07\\
0.01	7.73435954773869e-10\\
0.00562341325190349	1.78892421940189e-11\\
0.00316227766016838	1.78795655176405e-11\\
0.00177827941003892	1.78832389991264e-11\\
0.001	1.78807092926111e-11\\
0.000562341325190349	1.78749687988324e-11\\
0.000316227766016838	1.78752141615206e-11\\
0.000177827941003892	2.69526458060505e-11\\
};

\end{axis}
\end{tikzpicture}%

%% file: split_errphase_wrt_dt.tex
%
%
\definecolor{mycolor1}{rgb}{0.00000,0.44700,0.74100}%
\definecolor{mycolor2}{rgb}{0.85000,0.32500,0.09800}%
\definecolor{mycolor3}{rgb}{0.92900,0.69400,0.12500}%
\definecolor{mycolor4}{rgb}{0.49400,0.18400,0.55600}%
\definecolor{mycolor6}{rgb}{0.30100,0.74500,0.93300}%
\begin{tikzpicture}[scale=0.9]

\begin{axis}[%
width=4.754601cm,
height=3.75cm,
at={(0cm,0cm)},
scale only axis,
xmode=log,
xmin=0.0001,
xmax=0.1,
xminorticks=true,
xlabel={$h$},
ymode=log,
ymin=1e-12,
ymax=10,
yminorticks=true,
ylabel={$\mathcal{E}_{P,h}$},
axis x line*=bottom,
axis y line*=left
]
\addplot [color=mycolor1,solid,mark=+,mark options={solid},forget plot]
  table[row sep=crcr]{%
0.1	1.87309026001833\\
0.0562341325190349	0.798301014795155\\
0.0316227766016838	0.261249715585908\\
0.0177827941003892	0.0793376849134703\\
0.01	0.0279811666762367\\
0.00562341325190349	0.0162780742689965\\
0.00316227766016838	0.0093345651321766\\
0.00177827941003892	0.00530797293476788\\
0.001	0.00300375540494367\\
0.000562341325190349	0.0016951507579121\\
0.000316227766016838	0.0009551644364806\\
0.000177827941003892	0.00053773434385313\\
};
\addplot [color=mycolor2,solid,mark=o,mark options={solid},forget plot]
  table[row sep=crcr]{%
0.1	1.5951097297171\\
0.0562341325190349	0.638492356910844\\
0.0316227766016838	0.217822501863104\\
0.0177827941003892	0.0707238230233109\\
0.01	0.0225690200642836\\
0.00562341325190349	0.00715498442402973\\
0.00316227766016838	0.00226469143985011\\
0.00177827941003892	0.000716493301786844\\
0.001	0.000226569891081145\\
0.000562341325190349	7.1646302478078e-05\\
0.000316227766016838	2.26571921577287e-05\\
0.000177827941003892	7.16501195086487e-06\\
};
\addplot [color=mycolor4,solid,mark=diamond,mark options={solid},forget plot]
  table[row sep=crcr]{%
0.1	0.523878214568341\\
0.0562341325190349	0.2494978190919\\
0.0316227766016838	0.0692419907740724\\
0.0177827941003892	0.010502094683802\\
0.01	0.00122402023104902\\
0.00562341325190349	0.000128452580087184\\
0.00316227766016838	1.30366386158537e-05\\
0.00177827941003892	1.30986500253166e-06\\
0.001	1.31168531653556e-07\\
0.000562341325190349	1.3135470986063e-08\\
0.000316227766016838	1.38072266938559e-09\\
0.000177827941003892	2.24648333270291e-10\\
};
\addplot [color=mycolor6,solid,mark=square,mark options={solid},forget plot]
  table[row sep=crcr]{%
0.1	0.0118980996760246\\
0.0562341325190349	0.000941485760558232\\
0.0316227766016838	7.45905124090538e-05\\
0.0177827941003892	1.83092617365721e-06\\
0.01	1.32426250031002e-08\\
0.00562341325190349	2.0689017692093e-10\\
0.00316227766016838	1.94501733136311e-11\\
0.00177827941003892	3.17623422295484e-11\\
0.001	5.72489216703228e-11\\
0.000562341325190349	1.17399122037057e-10\\
0.000316227766016838	1.44109401764864e-10\\
0.000177827941003892	3.01663656997811e-10\\
};
\end{axis}
\end{tikzpicture}%

%% file: gauss_mass_wrt_dt.tex
%
%
\definecolor{mycolor1}{rgb}{0.00000,0.44700,0.74100}%
\definecolor{mycolor2}{rgb}{0.85000,0.32500,0.09800}%
\definecolor{mycolor3}{rgb}{0.92900,0.69400,0.12500}%
\definecolor{mycolor4}{rgb}{0.49400,0.18400,0.55600}%
\definecolor{mycolor5}{rgb}{0.46600,0.67400,0.18800}%
\begin{tikzpicture}[scale=0.9]

\begin{axis}[%
width=4.754601cm,
height=3.75cm,
at={(0cm,0cm)},
scale only axis,
xmode=log,
xmin=0.0001,
xmax=0.1,
xminorticks=true,
xlabel={$h$},
ymode=log,
ymin=1e-15,
ymax=1,
yminorticks=true,
ylabel={$\mathcal{E}_{M,h}$},
axis x line*=bottom,
axis y line*=left
]
\addplot [color=mycolor1,solid,mark=+,mark options={solid},forget plot]
  table[row sep=crcr]{%
0.1	0.000523174626026691\\
0.0562341325190349	5.70170188803898e-05\\
0.0316227766016838	6.49751699210999e-06\\
0.0177827941003892	6.66173775587531e-07\\
0.01	6.70500047705858e-08\\
0.00562341325190349	6.71940247887193e-09\\
0.00316227766016838	6.72391142764184e-10\\
0.00177827941003892	6.72645272814521e-11\\
0.001	6.7464922537397e-12\\
0.000562341325190349	7.07545133593613e-13\\
0.000316227766016838	4.79727368940531e-13\\
0.000177827941003892	2.07722727907367e-13\\
};
\addplot [color=mycolor2,solid,mark=o,mark options={solid},forget plot]
  table[row sep=crcr]{%
0.1	5.1458089182277e-05\\
0.0562341325190349	4.50839990784325e-07\\
0.0316227766016838	2.37080877152352e-09\\
0.0177827941003892	6.64202026712247e-12\\
0.01	7.21644966006353e-15\\
0.00562341325190349	2.14273043752655e-14\\
0.00316227766016838	2.60902410786912e-14\\
0.00177827941003892	5.80646641878958e-14\\
0.001	8.62643290133748e-14\\
0.000562341325190349	1.31117339208231e-13\\
0.000316227766016838	4.30655511252099e-13\\
0.000177827941003892	2.02615701994091e-13\\
};
\addplot [color=mycolor3,solid,mark=asterisk,mark options={solid},forget plot]
  table[row sep=crcr]{%
0.1	1.46885925089713e-07\\
0.0562341325190349	7.36692817859819e-10\\
0.0316227766016838	3.65363295173893e-12\\
0.0177827941003892	2.55351295663786e-15\\
0.01	6.77236045021346e-15\\
0.00562341325190349	1.765254609154e-14\\
0.00316227766016838	2.25375273998907e-14\\
0.00177827941003892	5.54001289287954e-14\\
0.001	8.20454815197991e-14\\
0.000562341325190349	1.27342580924506e-13\\
0.000316227766016838	4.29767332832399e-13\\
0.000177827941003892	1.98840943710366e-13\\
};
\addplot [color=mycolor4,solid,mark=diamond,mark options={solid},forget plot]
  table[row sep=crcr]{%
0.1	1.96374116967491e-09\\
0.0562341325190349	3.15303338993545e-12\\
0.0316227766016838	2.48689957516035e-14\\
0.0177827941003892	1.55431223447522e-14\\
0.01	1.3766765505352e-14\\
0.00562341325190349	4.99600361081321e-15\\
0.00316227766016838	9.21485110438881e-15\\
0.00177827941003892	3.65263375101677e-14\\
0.001	6.73905375947471e-14\\
0.000562341325190349	1.12909681604379e-13\\
0.000316227766016838	4.08006961549745e-13\\
0.000177827941003892	1.90847337933065e-13\\
};
\addplot [color=mycolor5,solid,mark=triangle,mark options={solid},forget plot]
  table[row sep=crcr]{%
0.1	1.30165878076127e-11\\
0.0562341325190349	2.28705943072782e-14\\
0.0316227766016838	3.27515792264421e-14\\
0.0177827941003892	2.76445533131664e-14\\
0.01	2.02060590481779e-14\\
0.00562341325190349	4.44089209850063e-15\\
0.00316227766016838	9.65894031423887e-15\\
0.00177827941003892	3.74145159298678e-14\\
0.001	6.6946448384897e-14\\
0.000562341325190349	1.10245146345278e-13\\
0.000316227766016838	4.17998968771372e-13\\
0.000177827941003892	1.9328982858724e-13\\
};
\end{axis}
\end{tikzpicture}%

%% file: gauss_ener_wrt_dt.tex
%
%
\definecolor{mycolor1}{rgb}{0.00000,0.44700,0.74100}%
\definecolor{mycolor2}{rgb}{0.85000,0.32500,0.09800}%
\definecolor{mycolor3}{rgb}{0.92900,0.69400,0.12500}%
\definecolor{mycolor4}{rgb}{0.49400,0.18400,0.55600}%
\definecolor{mycolor5}{rgb}{0.46600,0.67400,0.18800}%
\begin{tikzpicture}[scale=0.9]

\begin{axis}[%
width=4.754601cm,
height=3.75cm,
at={(0cm,0cm)},
scale only axis,
xmode=log,
xmin=0.0001,
xmax=0.1,
xminorticks=true,
xlabel={$h$},
ymode=log,
ymin=1e-15,
ymax=1,
yminorticks=true,
ylabel={$\mathcal{E}_{E,h}$},
axis x line*=bottom,
axis y line*=left
]
\addplot [color=mycolor1,solid,mark=+,mark options={solid},forget plot]
  table[row sep=crcr]{%
0.1	0.0201937086815006\\
0.0562341325190349	0.00144682379135614\\
0.0316227766016838	0.000137606667366022\\
0.0177827941003892	1.40397283541496e-05\\
0.01	1.4205757414988e-06\\
0.00562341325190349	1.42617341761378e-07\\
0.00316227766016838	1.42793064804053e-08\\
0.00177827941003892	1.4284134967742e-09\\
0.001	1.42735746719348e-10\\
0.000562341325190349	1.40676978263087e-11\\
0.000316227766016838	1.96616493578075e-12\\
0.000177827941003892	1.19912822840694e-12\\
};
\addplot [color=mycolor2,solid,mark=o,mark options={solid},forget plot]
  table[row sep=crcr]{%
0.1	0.000427532782837274\\
0.0562341325190349	5.08593949468999e-06\\
0.0316227766016838	1.74489853822676e-08\\
0.0177827941003892	1.65906720056108e-10\\
0.01	4.8223720099128e-14\\
0.00562341325190349	1.29819652295841e-13\\
0.00316227766016838	1.54805130463143e-13\\
0.00177827941003892	3.50670312459964e-13\\
0.001	5.36401524291026e-13\\
0.000562341325190349	8.18754899943891e-13\\
0.000316227766016838	2.72603797458912e-12\\
0.000177827941003892	1.24176079255255e-12\\
};
\addplot [color=mycolor3,solid,mark=asterisk,mark options={solid},forget plot]
  table[row sep=crcr]{%
0.1	2.13785905385124e-06\\
0.0562341325190349	1.30482714469314e-08\\
0.0316227766016838	3.75404188609357e-11\\
0.0177827941003892	7.94992487141422e-14\\
0.01	3.26733176033947e-14\\
0.00562341325190349	1.01863872528231e-13\\
0.00316227766016838	1.30169099542936e-13\\
0.00177827941003892	3.36867146199706e-13\\
0.001	5.12464387865009e-13\\
0.000562341325190349	7.97613341494636e-13\\
0.000316227766016838	2.72044681863559e-12\\
0.000177827941003892	1.21974561598555e-12\\
};
\addplot [color=mycolor4,solid,mark=diamond,mark options={solid},forget plot]
  table[row sep=crcr]{%
0.1	1.95532494239696e-08\\
0.0562341325190349	2.69663198874606e-11\\
0.0316227766016838	1.32964677519697e-13\\
0.0177827941003892	9.6097992951161e-14\\
0.01	8.89343243857108e-14\\
0.00562341325190349	2.39371364260165e-14\\
0.00316227766016838	4.45545240046292e-14\\
0.00177827941003892	2.18753976681552e-13\\
0.001	4.198608673848e-13\\
0.000562341325190349	7.05009821014426e-13\\
0.000316227766016838	2.58538545763332e-12\\
0.000177827941003892	1.16732852892128e-12\\
};
\addplot [color=mycolor5,solid,mark=triangle,mark options={solid},forget plot]
  table[row sep=crcr]{%
0.1	1.26272238566955e-10\\
0.0562341325190349	3.28305688645875e-13\\
0.0316227766016838	2.07222217527413e-13\\
0.0177827941003892	1.63890758887616e-13\\
0.01	1.2702407431908e-13\\
0.00562341325190349	2.39371364260165e-14\\
0.00316227766016838	4.76995492284854e-14\\
0.00177827941003892	2.23646238140884e-13\\
0.001	4.1549277679611e-13\\
0.000562341325190349	6.89284694895145e-13\\
0.000316227766016838	2.64566510775724e-12\\
0.000177827941003892	1.19004259998247e-12\\
};
\end{axis}
\end{tikzpicture}%

%% file: lawson_ener_wrt_dt.tex
%
%
\definecolor{mycolor1}{rgb}{0.00000,0.44700,0.74100}%
\definecolor{mycolor2}{rgb}{0.85000,0.32500,0.09800}%
\definecolor{mycolor3}{rgb}{0.92900,0.69400,0.12500}%
\definecolor{mycolor4}{rgb}{0.49400,0.18400,0.55600}%
\definecolor{mycolor5}{rgb}{0.46600,0.67400,0.18800}%
\begin{tikzpicture}[scale=0.9]

\begin{axis}[%
width=4.754601cm,
height=3.75cm,
at={(0cm,0cm)},
scale only axis,
xmode=log,
xmin=0.0001,
xmax=0.1,
xminorticks=true,
xlabel={$h$},
ymode=log,
ymin=1e-15,
ymax=1,
yminorticks=true,
ylabel={$\mathcal{E}_{E,h}$},
axis x line*=bottom,
axis y line*=left,
legend style={at={(0.03,0.97)},anchor=north west,legend cell align=left,align=left,draw=white!15!black}
]
\addplot [color=mycolor1,solid,mark=+,mark options={solid}]
  table[row sep=crcr]{%
0.1	0.070517508432005\\
0.0562341325190349	0.0150855154437275\\
0.0316227766016838	0.0022044480723283\\
0.0177827941003892	0.000257060631409899\\
0.01	2.71531329748644e-05\\
0.00562341325190349	2.76455988977689e-06\\
0.00316227766016838	2.78054884122933e-07\\
0.00177827941003892	2.78563178207351e-08\\
0.001	2.78725274554625e-09\\
0.000562341325190349	2.78815571723275e-10\\
0.000316227766016838	2.77997515717825e-11\\
0.000177827941003892	4.70792803648915e-12\\
};

\addplot [color=mycolor2,solid,mark=o,mark options={solid}]
  table[row sep=crcr]{%
0.1	0.111598057057056\\
0.0562341325190349	0.00669172126124561\\
0.0316227766016838	0.000152016116480583\\
0.0177827941003892	2.65286858011137e-07\\
0.01	2.0889519622288e-10\\
0.00562341325190349	1.90396332579782e-12\\
0.00316227766016838	8.12464849496179e-14\\
0.00177827941003892	2.60338199085872e-14\\
0.001	1.20384576624273e-13\\
0.000562341325190349	3.54688955801558e-13\\
0.000316227766016838	6.8823635315386e-13\\
0.000177827941003892	2.78457038847755e-12\\
};

\addplot [color=mycolor3,solid,mark=asterisk,mark options={solid}]
  table[row sep=crcr]{%
0.1	0.0123109558790268\\
0.0562341325190349	0.000214478337027028\\
0.0316227766016838	6.25043394528406e-07\\
0.0177827941003892	1.50286428110956e-10\\
0.01	4.5428142122367e-14\\
0.00562341325190349	8.73618117737827e-15\\
0.00316227766016838	6.42982934655041e-14\\
0.00177827941003892	2.63832671556824e-14\\
0.001	1.17588998647511e-13\\
0.000562341325190349	3.52417548695439e-13\\
0.000316227766016838	6.8893524764805e-13\\
0.000177827941003892	2.79068571530171e-12\\
};

\addplot [color=mycolor4,solid,mark=diamond,mark options={solid}]
  table[row sep=crcr]{%
0.1	0.0018676507181163\\
0.0562341325190349	2.69548366922135e-05\\
0.0316227766016838	2.27268686056284e-07\\
0.0177827941003892	1.86853461665108e-10\\
0.01	2.46360309202067e-14\\
0.00562341325190349	8.73618117737827e-15\\
0.00316227766016838	6.48224643361468e-14\\
0.00177827941003892	2.77810561440629e-14\\
0.001	1.20035129377177e-13\\
0.000562341325190349	3.52242825071892e-13\\
0.000316227766016838	6.89459418518693e-13\\
0.000177827941003892	2.78963737356043e-12\\
};

\addplot [color=mycolor5,solid,mark=triangle,mark options={solid}]
  table[row sep=crcr]{%
0.1	0.00031017648998372\\
0.0562341325190349	2.43728435781228e-06\\
0.0316227766016838	1.13964309901639e-08\\
0.0177827941003892	5.69214620793259e-12\\
0.01	1.64240206134711e-14\\
0.00562341325190349	9.95924654221123e-15\\
0.00316227766016838	6.5521358830337e-14\\
0.00177827941003892	2.8130503391158e-14\\
0.001	1.1916151125944e-13\\
0.000562341325190349	3.52592272318987e-13\\
0.000316227766016838	6.8823635315386e-13\\
0.000177827941003892	2.78719124283076e-12\\
};

\end{axis}
\end{tikzpicture}%

%% file: split_ener_wrt_dt.tex
%
%
\definecolor{mycolor1}{rgb}{0.00000,0.44700,0.74100}%
\definecolor{mycolor2}{rgb}{0.85000,0.32500,0.09800}%
\definecolor{mycolor3}{rgb}{0.92900,0.69400,0.12500}%
\definecolor{mycolor4}{rgb}{0.49400,0.18400,0.55600}%
\definecolor{mycolor6}{rgb}{0.30100,0.74500,0.93300}%
\begin{tikzpicture}[scale=0.9]

\begin{axis}[%
width=4.754601cm,
height=3.75cm,
at={(0cm,0cm)},
scale only axis,
xmode=log,
xmin=0.0001,
xmax=0.1,
xminorticks=true,
xlabel={$h$},
ymode=log,
ymin=1e-15,
ymax=1,
yminorticks=true,
ylabel={$\mathcal{E}_{E,h}$},
axis x line*=bottom,
axis y line*=left
]
\addplot [color=mycolor1,solid,mark=+,mark options={solid},forget plot]
  table[row sep=crcr]{%
0.1	1.36166793583039\\
0.0562341325190349	0.13588599752223\\
0.0316227766016838	0.0383296138655873\\
0.0177827941003892	0.0119622597015735\\
0.01	0.00375437372260611\\
0.00562341325190349	0.00118135243430094\\
0.00316227766016838	0.000372446791575382\\
0.00177827941003892	0.000117568207615927\\
0.001	3.71401352942058e-05\\
0.000562341325190349	1.17378547609058e-05\\
0.000316227766016838	3.71060185139852e-06\\
0.000177827941003892	1.17317482120217e-06\\
};
\addplot [color=mycolor2,solid,mark=o,mark options={solid},forget plot]
  table[row sep=crcr]{%
0.1	0.841776681000699\\
0.0562341325190349	0.0366015471644537\\
0.0316227766016838	0.00169852713111059\\
0.0177827941003892	0.000114940960236277\\
0.01	1.14652890582638e-05\\
0.00562341325190349	1.1489636941973e-06\\
0.00316227766016838	1.14951680786146e-07\\
0.00177827941003892	1.15007726183976e-08\\
0.001	1.14976793113664e-09\\
0.000562341325190349	1.14282878242744e-10\\
0.000316227766016838	1.09724688351636e-11\\
0.000177827941003892	8.91614650963226e-13\\
};
\addplot [color=mycolor4,solid,mark=diamond,mark options={solid},forget plot]
  table[row sep=crcr]{%
0.1	3.78867138868433\\
0.0562341325190349	0.0738477561737565\\
0.0316227766016838	0.00156985866829012\\
0.0177827941003892	4.15218357148449e-06\\
0.01	1.86261959298895e-08\\
0.00562341325190349	1.86135172848704e-10\\
0.00316227766016838	1.82254211722465e-12\\
0.00177827941003892	7.10076806097306e-13\\
0.001	9.42808672662663e-13\\
0.000562341325190349	1.49126612697847e-12\\
0.000316227766016838	8.39215036261311e-12\\
0.000177827941003892	1.13400873391076e-11\\
};
\addplot [color=mycolor6,solid,mark=square,mark options={solid},forget plot]
  table[row sep=crcr]{%
0.1	0.00899850713662092\\
0.0562341325190349	8.88478323114431e-05\\
0.0316227766016838	8.84375160032996e-07\\
0.0177827941003892	9.21513881595556e-10\\
0.01	5.44788258221309e-13\\
0.00562341325190349	1.20611717334884e-12\\
0.00316227766016838	2.16692237923691e-12\\
0.00177827941003892	3.20478070310944e-12\\
0.001	6.5378085459028e-12\\
0.000562341325190349	1.48995569980186e-11\\
0.000316227766016838	1.59363669801499e-11\\
0.000177827941003892	3.63505509845768e-11\\
};
\end{axis}
\end{tikzpicture}%

%% file: gauss_cpu_wrt_errphase.tex
%
%
\definecolor{mycolor1}{rgb}{0.00000,0.44700,0.74100}%
\definecolor{mycolor2}{rgb}{0.85000,0.32500,0.09800}%
\definecolor{mycolor3}{rgb}{0.92900,0.69400,0.12500}%
\definecolor{mycolor4}{rgb}{0.49400,0.18400,0.55600}%
\definecolor{mycolor5}{rgb}{0.46600,0.67400,0.18800}%
\begin{tikzpicture}[scale=0.9]

\begin{axis}[%
width=4.754601cm,
height=3.75cm,
at={(0cm,0cm)},
scale only axis,
xmode=log,
xmin=1e-12,
xmax=10,
xminorticks=true,
xlabel={$\mathcal{E}_{P,h}$},
ymode=log,
ymin=0.1,
ymax=1000,
yminorticks=true,
ylabel={CPU time},
axis x line*=bottom,
axis y line*=left,
grid=both
]
\addplot [color=mycolor1,solid,mark=+,mark options={solid},forget plot]
  table[row sep=crcr]{%
1.84557116883949	0.328174053\\
0.854659839575447	0.365371009\\
0.296255732876074	0.502674897\\
0.0961408191943907	0.670820865\\
0.0306557640598316	0.994640228\\
0.00971569216073048	1.505160506\\
0.00307488382984994	2.432359295\\
0.000972785095467171	3.958335111\\
0.000307611457316651	6.34697398399999\\
9.72730893833666e-05	10.0477732680001\\
3.07612955269807e-05	17.771279612\\
9.72781564205045e-06	27.946016024\\
};
\addplot [color=mycolor2,solid,mark=o,mark options={solid},forget plot]
  table[row sep=crcr]{%
0.0206578562543729	0.557592216\\
0.0018965971253223	0.638810398\\
0.000184889727845496	0.924043011\\
1.84075541001901e-05	1.358599943\\
1.83966489515784e-06	1.890271265\\
1.83875690309045e-07	2.788056384\\
1.83838390994948e-08	4.431067056\\
1.83455490448779e-09	6.99998537199999\\
1.78013931350978e-10	10.989256002\\
1.84182054965474e-11	19.307019064\\
2.62749707880861e-11	34.0237457649999\\
1.87358119529605e-11	52.5867491940001\\
};
\addplot [color=mycolor3,solid,mark=asterisk,mark options={solid},forget plot]
  table[row sep=crcr]{%
0.000104113083490949	0.695242178\\
3.46924799492737e-06	0.844285419\\
9.75706314906392e-08	1.229205774\\
2.81564462517891e-09	1.761198815\\
6.03476322824619e-11	2.622525585\\
1.78668283334528e-11	4.226334809\\
1.78793120866082e-11	6.78551690899999\\
1.78817098234705e-11	10.809797998\\
1.78802214977198e-11	16.89374449\\
1.78779835636396e-11	29.725667921\\
2.80640001297142e-11	52.567095799\\
1.87126209104355e-11	81.1794195870001\\
};
\addplot [color=mycolor4,solid,mark=diamond,mark options={solid},forget plot]
  table[row sep=crcr]{%
1.95975025292956e-06	0.916195853\\
4.69043548935922e-08	1.1437243\\
7.47730425142384e-10	1.588373687\\
1.97998503867999e-11	2.383441754\\
1.7866008510553e-11	3.663287967\\
1.78657825419057e-11	5.932114346\\
1.78793582225984e-11	9.48416689900001\\
1.78820304592522e-11	15.048209433\\
1.78806149103514e-11	25.6591066429999\\
1.787798555604e-11	41.5324919469999\\
2.72759731563889e-11	73.3152832650003\\
1.82576059116287e-11	112.187513657\\
};
\addplot [color=mycolor5,solid,mark=triangle,mark options={solid},forget plot]
  table[row sep=crcr]{%
5.37233329755416e-08	1.139763956\\
6.39038608078418e-10	1.361641929\\
1.81089773077742e-11	1.916955979\\
1.78653608612102e-11	2.90259713\\
1.7866527146341e-11	4.75911389300001\\
1.78656282966088e-11	7.539700938\\
1.78790630111594e-11	12.048331324\\
1.78820063289106e-11	19.008266958\\
1.78803180285659e-11	33.5663629030001\\
1.78779312278978e-11	52.2225415250001\\
2.72448712009838e-11	92.4421943679994\\
1.8150857704615e-11	141.338254376999\\
};
\end{axis}
\end{tikzpicture}%

%% file: lawson_cpu_wrt_errphase.tex
%
%
\definecolor{mycolor1}{rgb}{0.00000,0.44700,0.74100}%
\definecolor{mycolor2}{rgb}{0.85000,0.32500,0.09800}%
\definecolor{mycolor3}{rgb}{0.92900,0.69400,0.12500}%
\definecolor{mycolor4}{rgb}{0.49400,0.18400,0.55600}%
\definecolor{mycolor5}{rgb}{0.46600,0.67400,0.18800}%
\begin{tikzpicture}[scale=0.9]

\begin{axis}[%
width=4.754601cm,
height=3.75cm,
at={(0cm,0cm)},
scale only axis,
xmode=log,
xmin=1e-12,
xmax=10,
xminorticks=true,
xlabel={$  \mathcal{E}_{P,h}$},
ymode=log,
ymin=0.1,
ymax=1000,
yminorticks=true,
ylabel={CPU time},
axis x line*=bottom,
axis y line*=left,
legend style={legend cell align=left,align=left,draw=white!15!black},
grid=both
]
\addplot [color=mycolor1,solid,mark=+,mark options={solid}]
  table[row sep=crcr]{%
1.99697062079537	0.488077228\\
1.72977258988882	0.531012598\\
0.781322874394886	0.716684801\\
0.274041501642724	1.114468123\\
0.0894217269415249	1.47312851\\
0.0285418681884722	2.184927222\\
0.00905316108361094	3.413007539\\
0.00286610508064599	5.383366523\\
0.000906511662868099	8.68564070599998\\
0.000286677662672945	13.720143605\\
9.06599850424159e-05	24.4287055839999\\
2.8670152526422e-05	38.2086283859999\\
};

\addplot [color=mycolor2,solid,mark=o,mark options={solid}]
  table[row sep=crcr]{%
0.546383116130535	0.721768742\\
0.086429585015011	0.871906866\\
0.0102215160988692	1.263190042\\
0.00107665863494139	1.707635729\\
0.000109607125261287	2.547073565\\
1.10185325712024e-05	4.27548587799999\\
1.10380613542212e-06	6.229875457\\
1.1045986975924e-07	9.62148146100001\\
1.10465855504683e-08	16.962075928\\
1.1084758121233e-09	26.5870297990002\\
1.03415984199217e-10	47.0735980819998\\
3.59119986430842e-11	72.8071963179998\\
};

\addplot [color=mycolor3,solid,mark=asterisk,mark options={solid}]
  table[row sep=crcr]{%
0.084080306372219	1.027218289\\
0.00594354158902568	1.206500743\\
0.000275076621490442	1.769843406\\
1.00092814406047e-05	2.598145153\\
3.34821472290322e-07	3.870545335\\
1.07891455607122e-08	6.432007899\\
3.41385128151359e-10	10.226090992\\
1.79169987680997e-11	16.325877841\\
1.78809705832022e-11	29.189811039\\
1.78749034598323e-11	46.056203267\\
1.78750346931266e-11	81.8416295669998\\
2.69130211965167e-11	128.715827467999\\
};

\addplot [color=mycolor4,solid,mark=diamond,mark options={solid}]
  table[row sep=crcr]{%
0.0178006745529981	1.220959254\\
0.000798024145417313	1.567575132\\
3.93617425486699e-05	2.237776838\\
8.31124142997175e-07	3.367513965\\
5.21351394121515e-09	5.341224768\\
2.59684433879161e-11	8.61245523300001\\
1.78801690884268e-11	13.850753712\\
1.78832282264235e-11	22.20419608\\
1.78808683026381e-11	39.4363205179999\\
1.78749088206945e-11	62.5125740009999\\
1.78750230945903e-11	110.986429557\\
2.69388191015191e-11	172.964082147001\\
};

\addplot [color=mycolor5,solid,mark=triangle,mark options={solid}]
  table[row sep=crcr]{%
0.00456131470895431	1.507183643\\
0.000174326339488176	1.877257764\\
8.3667778358743e-06	2.717141223\\
1.43326425642903e-07	4.115026595\\
7.73435954773869e-10	6.661318297\\
1.78892421940189e-11	10.80925073\\
1.78795655176405e-11	17.390377922\\
1.78832389991264e-11	27.9595733229999\\
1.78807092926111e-11	49.350133308\\
1.78749687988324e-11	78.2026422930003\\
1.78752141615206e-11	138.475083828\\
2.69526458060505e-11	216.413550924999\\
};

\end{axis}
\end{tikzpicture}%

%% file: split_cpu_wrt_errphase.tex
%
%
\definecolor{mycolor1}{rgb}{0.00000,0.44700,0.74100}%
\definecolor{mycolor2}{rgb}{0.85000,0.32500,0.09800}%
\definecolor{mycolor3}{rgb}{0.92900,0.69400,0.12500}%
\definecolor{mycolor4}{rgb}{0.49400,0.18400,0.55600}%
\definecolor{mycolor6}{rgb}{0.30100,0.74500,0.93300}%
\begin{tikzpicture}[scale=0.9]

\begin{axis}[%
width=4.754601cm,
height=3.8125cm,
at={(0cm,0cm)},
scale only axis,
xmode=log,
xmin=1e-12,
xmax=10,
xminorticks=true,
xlabel={$\mathcal{E}_{P,h}$},
ymode=log,
ymin=0.1,
ymax=1000,
yminorticks=true,
ylabel={CPU time},
axis x line*=bottom,
axis y line*=left,
grid=both
]
\addplot [color=mycolor1,solid,mark=+,mark options={solid},forget plot]
  table[row sep=crcr]{%
1.87309026001833	0.235192842\\
0.798301014795155	0.295255082\\
0.261249715585908	0.543313691\\
0.0793376849134703	0.559619316\\
0.0279811666762367	0.885685204\\
0.0162780742689965	1.418827469\\
0.0093345651321766	1.588108281\\
0.00530797293476788	2.769436808\\
0.00300375540494367	4.89008192499999\\
0.0016951507579121	8.68550224699995\\
0.0009551644364806	15.443192954\\
0.00053773434385313	23.525553458\\
};
\addplot [color=mycolor2,solid,mark=o,mark options={solid},forget plot]
  table[row sep=crcr]{%
1.5951097297171	0.289778064\\
0.638492356910844	0.412300016\\
0.217822501863104	0.607916839\\
0.0707238230233109	0.969057364\\
0.0225690200642836	1.473442969\\
0.00715498442402973	1.634994834\\
0.00226469143985011	2.89551776500001\\
0.000716493301786844	4.995108253\\
0.000226569891081145	8.84447101700001\\
7.1646302478078e-05	15.604552938\\
2.26571921577287e-05	23.3814342510001\\
7.16501195086487e-06	41.7471752290001\\
};
\addplot [color=mycolor4,solid,mark=diamond,mark options={solid},forget plot]
  table[row sep=crcr]{%
0.523878214568341	0.560304113\\
0.2494978190919	0.762174357\\
0.0692419907740724	1.177767469\\
0.010502094683802	1.574432997\\
0.00122402023104902	2.436093596\\
0.000128452580087184	3.939701165\\
1.30366386158537e-05	5.552385993\\
1.30986500253166e-06	9.701221799\\
1.31168531653556e-07	15.811961589\\
1.3135470986063e-08	27.9448356929999\\
1.38072266938559e-09	45.303318866\\
2.24648333270291e-10	72.7190469269996\\
};
\addplot [color=mycolor6,solid,mark=square,mark options={solid},forget plot]
  table[row sep=crcr]{%
0.0118980996760246	1.707253168\\
0.000941485760558232	2.41483724\\
7.45905124090538e-05	3.784006992\\
1.83092617365721e-06	5.372549014\\
1.32426250031002e-08	8.488832275\\
2.0689017692093e-10	14.202644008\\
1.94501733136311e-11	21.787224319\\
3.17623422295484e-11	36.686250873\\
5.72489216703228e-11	60.3523145830001\\
1.17399122037057e-10	95.9837255569999\\
1.44109401764864e-10	160.646040712999\\
3.01663656997811e-10	275.450163347001\\
};
\end{axis}
\end{tikzpicture}%

%% file: err_phase_ERK_cas_test_pot_var.tex
%
%
\definecolor{mycolor1}{rgb}{0.00000,0.44700,0.74100}%
\definecolor{mycolor2}{rgb}{0.85000,0.32500,0.09800}%
\definecolor{mycolor3}{rgb}{0.92900,0.69400,0.12500}%
\definecolor{mycolor4}{rgb}{0.49400,0.18400,0.55600}%
\begin{tikzpicture}

\begin{axis}[%
width=4.342cm,
height=3.375cm,
at={(0cm,0cm)},
scale only axis,
xmode=log,
xmin=0.0001,
xmax=0.01,
xminorticks=true,
xlabel style={font=\color{white!15!black}},
xlabel={$h$},
ymode=log,
ymin=1e-16,
ymax=1,
yminorticks=true,
ylabel style={font=\color{white!15!black}},
ylabel={$\mathcal{E}_{P,h}$},
axis background/.style={fill=white},
axis x line*=bottom,
axis y line*=left
]
\addplot [color=mycolor1, mark=+, mark options={solid, mycolor1}, forget plot]
  table[row sep=crcr]{%
0.01	0.0363768313905969\\
0.00562341325190349	0.0115088356856232\\
0.00316227766016838	0.00364050104643232\\
0.00177827941003892	0.00115157915869679\\
0.001	0.000364122020266808\\
0.000562341325190349	0.00011513986315633\\
0.000316227766016838	3.64113218004101e-05\\
0.000177827941003892	1.15145464757506e-05\\
0.0001	3.64101010398779e-06\\
};
\addplot [color=mycolor2, mark=o, mark options={solid, mycolor2}, forget plot]
  table[row sep=crcr]{%
0.01	4.27385433215264e-06\\
0.00562341325190349	4.27560670178446e-07\\
0.00316227766016838	4.275941684778e-08\\
0.00177827941003892	4.26980447226232e-09\\
0.001	4.21442279841734e-10\\
0.000562341325190349	1.65574538687834e-10\\
0.000316227766016838	2.14658395306922e-10\\
0.000177827941003892	5.00100198069208e-10\\
0.0001	9.6332200328883e-10\\
};
\addplot [color=mycolor3, mark=asterisk, mark options={solid, mycolor3}, forget plot]
  table[row sep=crcr]{%
0.01	4.05586848713713e-10\\
0.00562341325190349	1.85439842164428e-11\\
0.00316227766016838	2.05371130425179e-11\\
0.00177827941003892	3.59386891702768e-11\\
0.001	5.9440165441334e-11\\
0.000562341325190349	1.30646276873558e-10\\
0.000316227766016838	1.95208689970589e-10\\
0.000177827941003892	2.3933597632888e-10\\
0.0001	6.30537361581784e-10\\
};
\addplot [color=mycolor4, mark=diamond, mark options={solid, mycolor4}, forget plot]
  table[row sep=crcr]{%
0.01	1.24162081251038e-11\\
0.00562341325190349	2.38264459300912e-11\\
0.00316227766016838	1.55483556612053e-11\\
0.00177827941003892	3.4767746058101e-11\\
0.001	7.18904233592851e-11\\
0.000562341325190349	1.66856091493615e-10\\
0.000316227766016838	2.14336118768207e-10\\
0.000177827941003892	4.99143275969805e-10\\
0.0001	9.63963095541815e-10\\
};
\end{axis}
\end{tikzpicture}%

%% file: err_phase_Lawson_cas_test_pot_var.tex
%
%
\definecolor{mycolor1}{rgb}{0.00000,0.44700,0.74100}%
\definecolor{mycolor2}{rgb}{0.85000,0.32500,0.09800}%
\definecolor{mycolor3}{rgb}{0.92900,0.69400,0.12500}%
\definecolor{mycolor4}{rgb}{0.49400,0.18400,0.55600}%
\begin{tikzpicture}

\begin{axis}[%
width=4.342cm,
height=3.375cm,
at={(0cm,0cm)},
scale only axis,
xmode=log,
xmin=0.0001,
xmax=0.01,
xminorticks=true,
xlabel style={font=\color{white!15!black}},
xlabel={$h$},
ymode=log,
ymin=1e-16,
ymax=1,
yminorticks=true,
ylabel style={font=\color{white!15!black}},
ylabel={$\mathcal{E}_{P,h}$},
axis background/.style={fill=white},
axis x line*=bottom,
axis y line*=left
]
\addplot [color=mycolor1, mark=+, mark options={solid, mycolor1}, forget plot]
  table[row sep=crcr]{%
0.01	0.0426963344591637\\
0.00562341325190349	0.0135016622037283\\
0.00316227766016838	0.00427018292533746\\
0.00177827941003892	0.00135067591680631\\
0.001	0.00042707185386922\\
0.000562341325190349	0.000135045194223639\\
0.000316227766016838	4.27060692008342e-05\\
0.000177827941003892	1.35053366838422e-05\\
0.0001	4.27073322142018e-06\\
};
\addplot [color=mycolor2, mark=o, mark options={solid, mycolor2}, forget plot]
  table[row sep=crcr]{%
0.01	4.26214890826682e-06\\
0.00562341325190349	4.26390476573141e-07\\
0.00316227766016838	4.26484558118234e-08\\
0.00177827941003892	4.26706112350454e-09\\
0.001	4.37145113039359e-10\\
0.000562341325190349	9.1156521095233e-11\\
0.000316227766016838	1.35721293143284e-10\\
0.000177827941003892	4.99950797155802e-10\\
0.0001	7.78439465599962e-10\\
};
\addplot [color=mycolor3, mark=asterisk, mark options={solid, mycolor3}, forget plot]
  table[row sep=crcr]{%
0.01	4.17669980269818e-10\\
0.00562341325190349	1.99913834743216e-11\\
0.00316227766016838	1.55423416987315e-11\\
0.00177827941003892	1.42648987942789e-11\\
0.001	5.04332279918596e-11\\
0.000562341325190349	8.45895708572392e-11\\
0.000316227766016838	1.35430653812331e-10\\
0.000177827941003892	4.9905924685886e-10\\
0.0001	7.77223234546121e-10\\
};
\addplot [color=mycolor4, mark=diamond, mark options={solid, mycolor4}, forget plot]
  table[row sep=crcr]{%
0.01	1.1677475886293e-11\\
0.00562341325190349	1.85220550955523e-11\\
0.00316227766016838	1.53805080590768e-11\\
0.00177827941003892	1.43004928549779e-11\\
0.001	5.0493544563122e-11\\
0.000562341325190349	8.45332778426912e-11\\
0.000316227766016838	1.35366197083678e-10\\
0.000177827941003892	5.00358669839987e-10\\
0.0001	7.78574081445516e-10\\
};
\end{axis}
\end{tikzpicture}%

%% file: err_mass_ERK_cas_test_pot_var.tex
%
%
\definecolor{mycolor1}{rgb}{0.00000,0.44700,0.74100}%
\definecolor{mycolor2}{rgb}{0.85000,0.32500,0.09800}%
\definecolor{mycolor3}{rgb}{0.92900,0.69400,0.12500}%
\definecolor{mycolor4}{rgb}{0.49400,0.18400,0.55600}%
\begin{tikzpicture}

\begin{axis}[%
width=4.342cm,
height=3.375cm,
at={(0cm,0cm)},
scale only axis,
xmode=log,
xmin=0.0001,
xmax=0.01,
xminorticks=true,
xlabel style={font=\color{white!15!black}},
xlabel={$h$},
ymode=log,
ymin=1e-16,
ymax=1,
yminorticks=true,
ylabel style={font=\color{white!15!black}},
ylabel={$\mathcal{E}_{M,h}$},
axis background/.style={fill=white},
axis x line*=bottom,
axis y line*=left
]
\addplot [color=mycolor1, mark=+, mark options={solid, mycolor1}, forget plot]
  table[row sep=crcr]{%
0.01	5.63415950150087e-06\\
0.00562341325190349	1.78219970817888e-06\\
0.00316227766016838	5.63631185093963e-07\\
0.00177827941003892	1.7824044952483e-07\\
0.001	5.6365039001474e-08\\
0.000562341325190349	1.78241961005433e-08\\
0.000316227766016838	5.63644792527394e-09\\
0.000177827941003892	1.78232637160507e-09\\
0.0001	5.63381933249465e-10\\
};
\addplot [color=mycolor2, mark=o, mark options={solid, mycolor2}, forget plot]
  table[row sep=crcr]{%
0.01	2.31283300541668e-11\\
0.00562341325190349	2.30949359535225e-12\\
0.00316227766016838	2.29385667520929e-13\\
0.00177827941003892	9.03760977754324e-14\\
0.001	1.57031782791624e-13\\
0.000562341325190349	3.51830703216676e-13\\
0.000316227766016838	5.28209861778407e-13\\
0.000177827941003892	6.50787413746552e-13\\
0.0001	1.74007117285807e-12\\
};
\addplot [color=mycolor3, mark=asterisk, mark options={solid, mycolor3}, forget plot]
  table[row sep=crcr]{%
0.01	2.34553802144451e-14\\
0.00562341325190349	4.47905000705222e-14\\
0.00316227766016838	4.889850451486e-14\\
0.00177827941003892	9.38215208577802e-14\\
0.001	1.5994714078438e-13\\
0.000562341325190349	3.53818447302646e-13\\
0.000316227766016838	5.29402508229989e-13\\
0.000177827941003892	6.47607023209e-13\\
0.0001	1.73742084741011e-12\\
};
\addplot [color=mycolor4, mark=diamond, mark options={solid, mycolor4}, forget plot]
  table[row sep=crcr]{%
0.01	1.64320177773513e-14\\
0.00562341325190349	3.78996539058265e-14\\
0.00316227766016838	4.20076583501643e-14\\
0.00177827941003892	8.60030607862986e-14\\
0.001	1.519961644405e-13\\
0.000562341325190349	3.46662568593154e-13\\
0.000316227766016838	5.22909210882487e-13\\
0.000177827941003892	6.42703921130274e-13\\
0.0001	1.73848097758929e-12\\
};
\end{axis}
\end{tikzpicture}%

%% file: err_mass_Lawson_cas_test_pot_var.tex
%
%
\definecolor{mycolor1}{rgb}{0.00000,0.44700,0.74100}%
\definecolor{mycolor2}{rgb}{0.85000,0.32500,0.09800}%
\definecolor{mycolor3}{rgb}{0.92900,0.69400,0.12500}%
\definecolor{mycolor4}{rgb}{0.49400,0.18400,0.55600}%
\begin{tikzpicture}

\begin{axis}[%
width=4.342cm,
height=3.375cm,
at={(0cm,0cm)},
scale only axis,
xmode=log,
xmin=0.0001,
xmax=0.01,
xminorticks=true,
xlabel style={font=\color{white!15!black}},
xlabel={$h$},
ymode=log,
ymin=1e-16,
ymax=1,
yminorticks=true,
ylabel style={font=\color{white!15!black}},
ylabel={$\mathcal{E}_{M,h}$},
axis background/.style={fill=white},
axis x line*=bottom,
axis y line*=left
]
\addplot [color=mycolor1, mark=+, mark options={solid, mycolor1}, forget plot]
  table[row sep=crcr]{%
0.01	6.89084616469572e-15\\
0.00562341325190349	1.45767899637794e-14\\
0.00316227766016838	2.92860961999568e-14\\
0.00177827941003892	2.10700873112812e-14\\
0.001	5.57893506795558e-14\\
0.000562341325190349	7.3944079998081e-14\\
0.000316227766016838	8.48104143347166e-14\\
0.000177827941003892	4.10800444433784e-13\\
0.0001	2.4051703440236e-13\\
};
\addplot [color=mycolor2, mark=o, mark options={solid, mycolor2}, forget plot]
  table[row sep=crcr]{%
0.01	5.96323225790976e-15\\
0.00562341325190349	1.39142086017894e-14\\
0.00316227766016838	2.76959009311809e-14\\
0.00177827941003892	1.86847944081172e-14\\
0.001	5.31390252315959e-14\\
0.000562341325190349	7.3811563725683e-14\\
0.000316227766016838	8.49429306071146e-14\\
0.000177827941003892	4.10667928161386e-13\\
0.0001	2.4051703440236e-13\\
};
\addplot [color=mycolor3, mark=asterisk, mark options={solid, mycolor3}, forget plot]
  table[row sep=crcr]{%
0.01	5.69819971311377e-15\\
0.00562341325190349	1.39142086017894e-14\\
0.00316227766016838	2.79609334759769e-14\\
0.00177827941003892	1.84197618633213e-14\\
0.001	5.32715415039939e-14\\
0.000562341325190349	7.4076596270479e-14\\
0.000316227766016838	8.52079631519106e-14\\
0.000177827941003892	4.10800444433784e-13\\
0.0001	2.40782066947156e-13\\
};
\addplot [color=mycolor4, mark=diamond, mark options={solid, mycolor4}, forget plot]
  table[row sep=crcr]{%
0.01	5.83071598551177e-15\\
0.00562341325190349	1.43117574189834e-14\\
0.00316227766016838	2.76959009311809e-14\\
0.00177827941003892	1.85522781357193e-14\\
0.001	5.28739926867999e-14\\
0.000562341325190349	7.3811563725683e-14\\
0.000316227766016838	8.52079631519106e-14\\
0.000177827941003892	4.10932960706182e-13\\
0.0001	2.40782066947156e-13\\
};
\end{axis}
\end{tikzpicture}%

%% file: gauss_errphase_wrt_dt_2D.tex
%
%
\definecolor{mycolor1}{rgb}{0.00000,0.44700,0.74100}%
\definecolor{mycolor2}{rgb}{0.85000,0.32500,0.09800}%
\definecolor{mycolor3}{rgb}{0.92900,0.69400,0.12500}%
\begin{tikzpicture}[scale=0.9]

\begin{axis}[%
width=4.754601cm,
height=3.75cm,
at={(0cm,0cm)},
scale only axis,
xmode=log,
xmin=0.001,
xmax=0.1,
xminorticks=true,
xlabel={$h$},
ymode=log,
ymin=1e-10,
ymax=10,
yminorticks=true,
ylabel={$\mathcal{E}_{P,h}$},
axis x line*=bottom,
axis y line*=left
]
\addplot [color=mycolor1,solid,mark=+,mark options={solid},forget plot]
  table[row sep=crcr]{%
0.1	2.91279796242156\\
0.0562341325190349	1.26414200637174\\
0.0316227766016838	0.452550054689343\\
0.0177827941003892	0.149734624581597\\
0.01	0.048133307293209\\
0.00562341325190349	0.0152865528238804\\
0.00316227766016838	0.00484217891203347\\
0.00177827941003892	0.0015327398905632\\
0.001	0.000484602844215684\\
};
\addplot [color=mycolor2,solid,mark=o,mark options={solid},forget plot]
  table[row sep=crcr]{%
0.1	0.00269595963579923\\
0.0562341325190349	0.000260323124956633\\
0.0316227766016838	2.56489986531622e-05\\
0.0177827941003892	2.52517897316708e-06\\
0.01	2.1837731481378e-07\\
0.00562341325190349	1.25622272231941e-08\\
0.00316227766016838	3.56397706521246e-08\\
0.00177827941003892	3.79782460402977e-08\\
0.001	3.8230072261986e-08\\
};
\addplot [color=mycolor3,solid,mark=asterisk,mark options={solid},forget plot]
  table[row sep=crcr]{%
0.1	6.37883900138203e-07\\
0.0562341325190349	1.73826162408889e-08\\
0.0316227766016838	3.74670114296653e-08\\
0.0177827941003892	3.81130891248458e-08\\
0.01	3.81992252630183e-08\\
0.00562341325190349	3.81889324925336e-08\\
0.00316227766016838	3.82029661401289e-08\\
0.00177827941003892	3.82343747155712e-08\\
0.001	3.82551777648394e-08\\
};
\end{axis}
\end{tikzpicture}%

%% file: lawson_errphase_wrt_dt_2D.tex
%
%
\definecolor{mycolor1}{rgb}{0.00000,0.44700,0.74100}%
\definecolor{mycolor2}{rgb}{0.85000,0.32500,0.09800}%
\definecolor{mycolor3}{rgb}{0.92900,0.69400,0.12500}%
\begin{tikzpicture}[scale=0.9]

\begin{axis}[%
width=4.754601cm,
height=3.75cm,
at={(0cm,0cm)},
scale only axis,
xmode=log,
xmin=0.001,
xmax=0.1,
xminorticks=true,
xlabel={$h$},
ymode=log,
ymin=1e-10,
ymax=10,
yminorticks=true,
ylabel={$\mathcal{E}_{P,h}$},
axis x line*=bottom,
axis y line*=left
]
\addplot [color=mycolor1,solid,mark=+,mark options={solid},forget plot]
  table[row sep=crcr]{%
0.1	6.24493087394186\\
0.0562341325190349	4.9114905948761\\
0.0316227766016838	2.81092852777815\\
0.0177827941003892	1.20150753499128\\
0.01	0.429381161853651\\
0.00562341325190349	0.141549560501359\\
0.00316227766016838	0.0453896645572215\\
0.00177827941003892	0.0144242905500396\\
0.001	0.00456637927179504\\
};
\addplot [color=mycolor2,solid,mark=o,mark options={solid},forget plot]
  table[row sep=crcr]{%
0.1	1.44472088132077\\
0.0562341325190349	0.21140875975344\\
0.0316227766016838	0.022751417782964\\
0.0177827941003892	0.00230791264230029\\
0.01	0.000231522566685134\\
0.00562341325190349	2.30927024084553e-05\\
0.00316227766016838	2.27341885673319e-06\\
0.00177827941003892	1.92925122366831e-07\\
0.001	1.51341715145256e-08\\
};
\addplot [color=mycolor3,solid,mark=asterisk,mark options={solid},forget plot]
  table[row sep=crcr]{%
0.1	0.126678735312914\\
0.0562341325190349	0.00585195012566455\\
0.0316227766016838	0.000211347045003773\\
0.0177827941003892	6.96372036257736e-06\\
0.01	1.864660963534e-07\\
0.00562341325190349	3.11646053879523e-08\\
0.00316227766016838	3.8024993048283e-08\\
0.00177827941003892	3.8228825187222e-08\\
0.001	3.82218264964175e-08\\
};
\end{axis}
\end{tikzpicture}%

%% file: split_errphase_wrt_dt_2D.tex
%
%
\definecolor{mycolor1}{rgb}{0.00000,0.44700,0.74100}%
\definecolor{mycolor2}{rgb}{0.85000,0.32500,0.09800}%
\definecolor{mycolor3}{rgb}{0.92900,0.69400,0.12500}%
\definecolor{mycolor4}{rgb}{0.49400,0.18400,0.55600}%
\begin{tikzpicture}[scale=0.9]

\begin{axis}[%
width=4.754601cm,
height=3.75cm,
at={(0cm,0cm)},
scale only axis,
xmode=log,
xmin=0.001,
xmax=0.1,
xminorticks=true,
xlabel={$h$},
ymode=log,
ymin=1e-10,
ymax=10,
yminorticks=true,
ylabel={$\mathcal{E}_{P,h}$},
axis x line*=bottom,
axis y line*=left
]
\addplot [color=mycolor1,solid,mark=+,mark options={solid},forget plot]
  table[row sep=crcr]{%
0.1	3.31097002387137\\
0.0562341325190349	1.31584756420413\\
0.0316227766016838	2.34261741693371\\
0.0177827941003892	2.07329811886577\\
0.01	1.35153447182348\\
0.00562341325190349	0.799894098624813\\
0.00316227766016838	0.458986978472355\\
0.00177827941003892	0.260459359160299\\
0.001	0.147041280656367\\
};
\addplot [color=mycolor2,solid,mark=o,mark options={solid},forget plot]
  table[row sep=crcr]{%
0.1	5.12505981224286\\
0.0562341325190349	3.05410818481987\\
0.0316227766016838	1.33851688493136\\
0.0177827941003892	0.48484447693227\\
0.01	0.161088069040281\\
0.00562341325190349	0.0517241977532673\\
0.00316227766016838	0.0164422111187166\\
0.00177827941003892	0.00521055356188327\\
0.001	0.00164804955758433\\
};
\addplot [color=mycolor3,solid,mark=diamond,mark options={solid},forget plot]
  table[row sep=crcr]{%
0.1	6.47963984234833\\
0.0562341325190349	1.91932419631739\\
0.0316227766016838	0.233892799842726\\
0.0177827941003892	0.0262722288561889\\
0.01	0.00274510795106126\\
0.00562341325190349	0.000278141549846207\\
0.00316227766016838	2.79716668965111e-05\\
0.00177827941003892	2.83684752745736e-06\\
0.001	3.18271598632587e-07\\
};
\addplot [color=mycolor4,solid,mark=square,mark options={solid},forget plot]
  table[row sep=crcr]{%
0.1	0.00460837248789056\\
0.0562341325190349	0.00019601941999269\\
0.0316227766016838	5.91979935836758e-06\\
0.0177827941003892	1.09910674872624e-07\\
0.01	3.37328654772121e-08\\
0.00562341325190349	3.81473451967411e-08\\
0.00316227766016838	3.83781389013613e-08\\
0.00177827941003892	3.85192537434476e-08\\
0.001	3.88023671477655e-08\\
};
\end{axis}
\end{tikzpicture}%